\documentclass[12pt]{amsart}
\usepackage[toc,page]{appendix}
\usepackage{amssymb, latexsym, euscript}
\usepackage{amssymb}
\usepackage{latexsym}
\usepackage{amsmath}

\def\sD{{\mathfrak D}}      
   \def\sH{{\mathfrak H}}   
\def\sJ{{\mathfrak J}}   \def\sK{{\mathfrak K}}   \def\sL{{\mathfrak L}}
\def\sM{{\mathfrak M}}   \def\sN{{\mathfrak N}}   
\def\sP{{\mathfrak P}}

      \def\dC{{\mathbb C}}
\def\dD{{\mathbb D}}

   \def\dN{{\mathbb N}}   
      \def\dR{{\mathbb R}}
   \def\dT{{\mathbb T}}   
      
   \def\dZ{{\mathbb Z}}

\def\cA{{\mathcal A}}   \def\cB{{\mathcal B}}   \def\cC{{\mathcal C}}
\def\cD{{\mathcal D}}      
      
\def\cJ{{\mathcal J}}      
\def\cM{{\mathcal M}}   \def\cN{{\mathcal N}}   
   \def\cQ{{\mathcal Q}}   \def\cR{{\mathcal R}}
\def\cS{{\mathcal S}}   \def\cT{{\mathcal T}}   \def\cU{{\mathcal U}}
\def\cV{{\mathcal V}}   \def\cW{{\mathcal W}}

   \def\bB{{\mathbf B}}

\def\bS{{\mathbf S}}

\def\cRS{\mathcal{RS}}

\def\d1{{\mathcal D}}

\def\wt{\widetilde}

\def\d1{{\mathcal D}}
\def\CM{{\cM}}

\def\wt#1{{{\widetilde #1} }}
\def\wh#1{{{\widehat #1} }}

\def\bm\chi{\mbox{\boldmath$\chi$}}

\def\RE{{\rm Re\,}}
\def\IM{{\rm Im\,}}

\def\ran{{\rm ran\,}}
\def\cran{{\rm \overline{ran}\,}}
\def\dom{{\rm dom\,}}
\def\mul{{\rm mul\,}}

\def\cdom{{\rm \overline{dom}\,}}

\let\xker=\ker \def\ker{{\xker\,}}

\def\cspan{{\rm \overline{span}\, }}

\def\cmr{{\dC \backslash \dR}}
\def\uphar{{\upharpoonright\,}}

\def\RE{{\rm Re\,}}
\def\IM{{\rm Im\,}}

\def\wt{\widetilde}

\def\f{\varphi}

\newtheorem{theorem}{Theorem}[section]
\newtheorem{lemma}[theorem]{Lemma}
\newtheorem{proposition}[theorem]{Proposition}

\newtheorem{definition}[theorem]{Definition}
\newtheorem{remark}[theorem]{Remark}

\def\RE{{\rm Re\,}}
\def\IM{{\rm Im\,}}

\def\wt{\widetilde}
\def\wh{\widehat}

\def\f{\varphi}

\def\uphar{{\upharpoonright\,}}

\numberwithin{equation}{section}
\begin{document}

\title[Stieltjes and inverse Stieltjes families]
{Stieltjes and inverse Stieltjes holomorphic families of linear relations and their representations}
\author[Yury Arlinski\u{\i}]{Yu.M. Arlinski\u{\i}}
\address{Volodymyr Dahl East Ukrainian National University \\
pr. Central 59-A, Severodonetsk, 93400, Ukraine}
\email{yury.arlinskii@gmail.com}

\author[Seppo Hassi]{S. Hassi}
\address{Department of Mathematics and Statistics  \\
University of Vaasa  \\
P.O. Box 700  \\
65101 Vaasa  \\
Finland} \email{sha@uwasa.fi}

\thanks{This research was partially supported by a grant from the Vilho,
Yrj\"o and Kalle V\"ais\"al\"a Foundation of the Finnish Academy of
Science and Letters. Yu.M.~Arlinski\u{\i} also gratefully acknowledges
financial support from the University of Vaasa.}

\vskip 1truecm
\thispagestyle{empty}
\baselineskip=12pt

\date{ \today;  Filename: \jobname}
\begin{abstract}
We study analytic and geometric properties of Stieltjes and inverse Stieltjes families defined on a separable Hilbert space and establish various minimal representations for them by means of compressed resolvents of various types of linear relations. Also attention is paid to some new peculiar properties of Stieltjes and inverse Stieltjes families,
including an analog for the notion of inner functions which will be characterized in an explicit manner. In addition, families which admit different types
of scale invariance properties are described. Two transformers that naturally appear in the Stieltjes and inverse Stieltjes classes are introduced and their fixed points
are identified.
\end{abstract}
\maketitle


\section{Introduction}

The main objects in this paper are the general classes of operator-valued Stieltjes and inverse Stieltjes functions or, more generally, Stieltjes and inverse Stieltjes families of linear relations, whose values depend on a complex variable $\lambda\in \dC\setminus{\dR_+}$ and whose values are linear relations acting on some fixed Hilbert space $\sM$. Functions in these classes are Nevanlinna functions or, more generally, Nevanlinna families
which admit holomorphic continuation to the negative semi-axis $\dR_-$. Such functions typically appear in boundary value problems when modeling physical phenomena and they offer an important analytic tool for the spectral analysis of associated nonnegative selfadjoint operators; see e.g.  \cite{AG,Kac1,KacK,KacK2} for some classical treatments in the case of scalar functions. In particular, functions belonging to the inverse Stieltjes class are relevant in the spectral analysis of the Friedrichs extension $A_F$ of a nonnegative operator $A$, cf. \cite{KacK,BHS2009CAOT,ArlHassi_2015OT}.

In order to introduce a general definition of Stieltjes and inverse Stieltjes classes, first recall the definition of a Nevanlinna family;
cf. e.g. \cite{DHMS06,DHMS09,DdeS,LT} and the references therein.

\begin{definition}\label{Nevanlinna}
A family of linear relations $\CM(\lambda)$, $\lambda \in \cmr$, in a
Hilbert space $\sM$ is called a \textit{Nevanlinna family} if:
\begin{enumerate}
\def\labelenumi{\rm (\roman{enumi})}
\item $\CM(\lambda)$ is maximal dissipative for every $\lambda \in \dC_+$
(resp. accumulative for every $\lambda \in \dC_-$);
\item $\CM(\lambda)^*=\cM(\bar \lambda)$, $\lambda \in \cmr$;
\item for some (and hence for all) $\mu \in \dC_+\, (\dC_-)$ the
      operator family $(\CM(\lambda)+\mu)^{-1} (\in \bB(\sM))$ is
      holomorphic for all $\lambda \in \dC_+\, (\dC_-)$.
\end{enumerate}
\end{definition}

The class of all Nevanlinna families in a Hilbert space $\sM$ is denoted by $\wt R(\sM)$.
Each Nevanlinna family $\CM$ admits a unique decomposition to its operator part $M_s(\lambda)$,
$\lambda\in\cmr$, and the constant multi-valued part $M_\infty$:
\[
  \CM(\lambda)=M_s(\lambda)\oplus M_\infty,\quad
  M_\infty=\{0\}\times\mul \CM(\lambda).
\]
Here $M_s(\lambda)$ is a Nevanlinna family of densely defined
operators in the subspace $\sM \ominus \mul \CM(\lambda)$.

The main objective of the present paper is a study of \textit{Stieltjes} and  \textit{inverse Stieltjes} holomorphic families of linear relations,
which are defined as follows.

\begin{definition}\label{invStieltjes}
A family of linear relations $\CM(\lambda)$, $\lambda \in \dC\setminus\dR_+$, in a
Hilbert space $\sM$ is said to be a \textit{Stieltjes family} (respectively, \textit{inverse Stieltjes family})
if it is a Nevanlinna family and, moreover,
\begin{enumerate}
\def\labelenumi{\rm (\roman{enumi})}
\item $\cM(x)\ge 0$ for all $x<0$ (respectively, $\cM(x)\le 0$ for all $x<0$),
\item for some (and hence for all) $\mu \in \dC_+\,(\dC_-)$ the
      operator family $(\CM(\lambda)+\mu)^{-1} (\in \bB(\sM))$ with $\lambda \in \dC_+\,(\dC_-)$
      admits an analytic continuation to the semi-axis $\lambda \in (-\infty,0)$.
\end{enumerate}
\end{definition}

The classes of all Stieltjes and inverse Stieltjes families in a Hilbert space $\sM$
are denoted by $\wt \cS(\sM)$ and $\wt\cS^{(-1)}(\sM)$, respectively.

By definition, a Stieltjes family (resp. inverse Stieltjes family) is a Nevanlinna family, which
admits a holomorphic continuation to the negative semi-axis
$(-\infty,0)$ and whose values for $x\in (-\infty,0)$ are nonnegative (resp. nonpositive) and in view of (ii) also selfadjoint relations.
As in the case of scalar functions these classes are connected to each other via appropriate inversions:
if $\cM(\lambda)$ is a Stieltjes (resp. inverse Stieltjes) family, then
\begin{enumerate}
\item[(a)] $-\cM(1/\lambda)$ is an inverse Stieltjes (resp. Stieltjes) family,
\item[(b)] $-\cM^{-1}(\lambda)$ is an inverse Stieltjes (resp. Stieltjes) family.
\end{enumerate}
Here item (a) is clear. To see (b) apply the formula
\begin{equation}\label{invres}
 \left(H^{-1}-\frac{1}{z}\right)^{-1}=-z-z^2\left(H-z\right)^{-1} \quad
 \left(z\in\rho(H)\quad\Longleftrightarrow\quad z^{-1}\in\rho(H^{-1})\right)
\end{equation}
to $H=M(\lambda)$ and $z=-\mu$ ($\lambda,\mu\in\dC_+$ or $\lambda,\mu\in\dC_-$).
Hence, if $(\CM(\lambda)+\mu)^{-1} (\in \bB(\sM))$ admits an analytic continuation to
$\lambda \in (-\infty,0)$ the same is true for $(\CM(\lambda)^{-1}+\mu^{-1})^{-1} (\in \bB(\sM))$.
This shows equivalence in property (ii) of Definition~\ref{invStieltjes}; the equivalence in property (i)
is clear from $M(x)\geq 0$ $\Leftrightarrow$  $-M(x)^{-1}\leq 0$. In fact, another proof for (b) is contained
in Lemma~\ref{TH1} appearing in Section~\ref{sec3}.

An important example of a Nevanlinna family is obtained by compressing the resolvent $(\wt A-\lambda)^{-1}$
of a selfadjoint relation $\wt A$ in a Hilbert space $\sH$ to some subspace $\sM$ of $\sH$:
\begin{equation}\label{compres1}
 P_\sM (\wt A-\lambda)^{-1}\uphar \sM \in \wt R(\sM),
\end{equation}
which is an operator valued Nevanlinna function.
If, in addition, $\wt A$ is nonnegative, then $P_\sM (\wt A-\lambda)^{-1}$ is a Stieltjes family of bounded operators.

A selfadjoint relation $\wt A$ in the orthogonal sum $\sH=\sM\oplus\sK$ is called \textit{minimal with respect to $\sM$} or, $\sM$-\textit{minimal} for short (see \cite[page 5366]{DHMS06}), if
\begin{equation}\label{amini}
 \sH=\cspan\left\{\sM+(\wt A-\lambda I)^{-1}\sM: \lambda\in\cmr\right\}.
\end{equation}
The set $\cmr$ in \eqref{amini} can be replaced by a union of two open sets, with one open set from $\dC_+$ and the other one from $\dC_-$.
Moreover, this definition of minimality can be extended to non-selfadjoint relations $\wt A$ in $\sH=\sM\oplus\sK$ with $\rho(\wt A)\neq \emptyset$
by replacing the set $\cmr$ in \eqref{amini} by the resolvent set $\rho(\wt A)$, or by a union of open sets, including one open set
from each connected component of $\rho(\wt A)$. In what follows this minimality condition in this more general form is applied to nonnegative and, more generally, to maximal accretive relations with $\lambda$ in \eqref{amini} taken from the left half-plane in $\dC$.

Two selfadjoint relations $\wt A^{(1)}$ and $\wt A^{(2)}$ in the Hilbert spaces $\sM\oplus\sK^{(1)}$ and $\sM\oplus\sK^{(2)}$, respectively, are
said to be \textit{unitarily equivalent} if there exists a unitary operator $\cV$ acting from $\sK^{(1)}$ onto $\sK^{(2)}$, such that
\[
\wt A^{(2)}=\left\{\left\{\begin{bmatrix}\f\cr \cV f\end{bmatrix},\begin{bmatrix}\f'\cr \cV f'\end{bmatrix}\right\}:\left\{\begin{bmatrix}\f\cr f\end{bmatrix},\begin{bmatrix}\f'\cr f'\end{bmatrix}\right\}\in\wt A^{(1)}\right\},\; \f,\f'\in\sM,\;f,f'\in\sK.
\]
In \cite{DHMS06} in the context of the Weyl families of boundary relations it has been proven that
for an arbitrary Nevanlinna family $\cM$ in the Hilbert space $\sM$ there exists (up to unitary equivalence)
a unique selfadjoint relation $\wt A$ in the Hilbert space $\sM\oplus\sK$ which is $\sM$-minimal such that the formula
\begin{equation}
\label{opexpr1}
 P_\sM(\wt A-\lambda I)^{-1}\uphar\sM=-(\cM(\lambda)+\lambda I)^{-1},\; \lambda\in\cmr,
\end{equation}
holds; see the proof of \cite[Theorem~3.9]{DHMS06}.
Inverting the formula \eqref{opexpr1} in the relation sense leads to an equivalent expression
\begin{equation}
\label{opexpr1B}
\cM(\lambda)=-\left(P_\sM\left(\wt A-\lambda I\right)^{-1}\uphar\sM\right)^{-1}-\lambda I_\sM ,\; \lambda\in\cmr.
\end{equation}
If, in addition, $\wt A$ is nonnegative, then it follows from the general property (b) when applied to the
Stietjes function \eqref{compres1} that the Nevanlinna family $\cM(\cdot)$ in \eqref{opexpr1B} is, in fact,
an inverse Stieltjes family.
As shown in Theorem \ref{obratno1}, this representation describes all inverse Stieltjes families.

The formula \eqref{opexpr1} is closely related with the description of generalized resolvents by A.V. Shtraus \cite{shtraus1954};
cf. \cite[Theorem~5.2]{DHMS09}. It should be also noted that the literature related to representations of operator valued
Nevanlinna functions and Nevanlinna families as compressed resolvents of selfadjoint exit space extensions has been
studied extensively and various related contributions can be found e.g. in
\cite{ArlBelTsek2011, ArlHassi_2015, AHS2, ArlKlotz2010, BMNW2017, DdeS, KalWo, KL1, KL2, KL, McKelvey, LT}.
Some subclasses of Stieltjes and inverse Stieltjes matrix-valued functions have been considered in \cite{ArlBelTsek2011}, where the realizations of Nevanlinna matrix-valued functions as the impedance functions of singular $L$-systems are studied.

In the present paper the special attention is in characteristic properties as well as in various descriptions of
Stieltjes and inverse Stieltjes families. Since these families are special type of Nevanlinna families it is of interest to
characterize those selfadjoint relations which yield to their representations by means of compressed resolvents analogous to
the formula \eqref{opexpr1}. A closer investigation of the properties of these families is obtained by using suitable
linear fractional transformations of the (graphs) of selfadjoint relations. In particular, the following two transformations
will frequently appear in this paper:
the transformation $\sP_\sM$ defined in $(\sM\oplus\sK)^2$ by the formula
\begin{equation}
\label{ghtjha21}
 \sP_\sM:
 \left\{ \begin{bmatrix} \f \\ f \end{bmatrix},
        \begin{bmatrix} \f' \\ f' \end{bmatrix}
 \right\}
 {\mapsto}
 \left\{ \begin{bmatrix} \f'\\ f  \end{bmatrix},
         \begin{bmatrix}  \f \\ f'\end{bmatrix}
          \right\},\quad \f,\f'\in\sM,\; f,f'\in\sK,
 \end{equation}
and the transformation $\sJ_\sM$, which is defined in $(\sM\oplus\sK)^2$ by the formula
\begin{equation}
\label{ghtjha1}
\sJ_\sM:
 \left\{ \begin{bmatrix} \f \\ f \end{bmatrix},
        \begin{bmatrix} \f' \\ f' \end{bmatrix}
 \right\}
 {\mapsto}
 \left\{ \begin{bmatrix} -i\f'\\ f  \end{bmatrix},
         \begin{bmatrix}  i\f \\ f'\end{bmatrix}
 \right\},\quad   \f,\f'\in\sM,\; f,f'\in\sK.
 \end{equation}
Each of these transformations is an involution in $(\sM\oplus\sK)^2$:
$\left(\sJ_\sM\right)^2=\left(\sP_\sM\right)^2=I_{(\sM\oplus\sK)^2}$.

On the other hand, in establishing the main results of the present paper also various relationships between selfadjoint
contractions, nonnegative selfadjoint relations, resolvents of selfadjoint relations,
and transfer functions of passive selfadjoint system (studied recently in \cite{ArlHassi_2018}, cf. Appendix \ref{AppendA})
will be used.

Main results in this paper can be briefly described as follows:
\begin{itemize}
\item One-to-one correspondences between the classes of Stieltjes/inverse Stieltjes families in the Hilbert space $\sM$ and transfer functions from the
combined Nevanlinna-Schur class $\cRS(\sM)$ (being recently studied in \cite{ArlHassi_2018}) of discrete-time passive selfadjoint systems are established (see Lemma \ref{TH1}).

\item It is proved (see Theorems \ref{obratno1}, \ref{stieltacc}, and~\ref{zwanzig}) that inverse Stieltjes families in $\sM$ admit $\sM$-minimal representations of the form \eqref{opexpr1} by means of compressed resolvents of \textit{nonnegative selfadjoint relations} \textit{for all} $\lambda\in\dC\setminus\dR_+$, while for Stieltjes families there are $\sM$-minimal representations of the form \eqref{opexpr1} by means of the compressed resolvents of a selfadjoint relation $\wh A$ for all $\lambda\in\cmr$ analogous to \eqref{opexpr1} by means of \textit{maximal accretive} relations $\wh B$ for $\lambda$ in the open left half-plane. Here the selfadjoint relations $\wh A$ and $\wh B$ are connected with a nonnegative selfadjoint relation $\wt A$ via the transformations \eqref{ghtjha1}, \eqref{ghtjha21}: $\sP_\sM(\wh B)=\sJ_\sM(\wh A)=\wt A$.

 \item In Section \ref{innerfamily} \textit{inner functions in the Stieltjes and inverse Stieltjes classes} are introduced and characterized.

 \item In Section \ref{scaleinv11} all those Stieltjes and inverse Stieltjes families $\cM$, which admit the following scaling property
\[
 \cM(c\lambda)=c^p\cM(\lambda)\; \forall\lambda\in\cmr
\]
for some $c\in\dR_+$ and for some $p\in\{0,1,-1\}$ will be described. We call such families \textit{scale invariant}.

\item In Section \ref{sec4.3} two transformers $\Phi_+$ and $\Phi_-$ appearing in the classes of Stieltjes and inverse Stieltjes families are shortly studied.
In particular, we identify the fixed points of the mappings
\[
 \wt\cS(\sM)\ni\cQ(\lambda)\mapsto -\cfrac{\cQ(\lambda)^{-1}}{\lambda}\in\wt\cS(\sM)
\]
and
\[
 \wt\cS^{(-1)}(\sM)\ni\cR(\lambda)\mapsto -\lambda \cR(\lambda)^{-1}\in\wt\cS^{(-1)}(\sM)
\]
and offer two appropriate realizations for them; these will be also used as examples for the representation results given in Section \ref{sec3}.
\end{itemize}

{\bf Notations.}
We use the symbols $\dom T$, $\ran T$, $\ker T$ for
the domain, the range, and the null-subspace of a linear operator
$T$. The closures of $\dom T$, $\ran T$ are denoted by $\cdom T$,
$\cran T$, respectively. The identity operator in a Hilbert space
$\sH$ is denoted by  $I$ and sometimes by $I_\sH$. If $\sL$ is a
subspace, i.e., a closed linear subset of $\sH$, the orthogonal
projection in $\sH$ onto $\sL$ is denoted by $P_\sL.$ The notation
$T\uphar \cN$ means the restriction of a linear operator $T$ on the
set $\cN\subset\dom T$. The resolvent set of $T$ is denoted by
$\rho(T)$.
The linear space of bounded
operators acting between Hilbert spaces $\sH$ and $\sK$ is denoted
by $\bB(\sH,\sK)$ and the Banach algebra $\bB(\sH,\sH)$ by
$\bB(\sH).$ $\dC_{+}/\dC_-$ denotes the open upper/lower half-plane of $\dC$, $\dR$ denotes the set of real numbers, $\dR_+:=[0,+\infty)$,
$\dZ$ and $\dN$ are the sets of integers and natural numbers, $\dN_0:=\dN\cup\{0\}$, $\dD=\{z\in\dC:|z|<1\}$ is the unit disk, $\dT=\{\zeta\in\dC:|\zeta|=1\}$ is the unit circle. By ${\bf S}(\sH_1,\sH_2)$ we denote the \textit{Schur class} (the set of all holomorphic and contractive
$\bB(\sH_1,\sH_2)$-valued functions on the unit disk) and ${\bf S}(\sH):=\bS(\sH,\sH).$
For a contraction
$T\in\bB(\sH,\sK)$ the defect operator $(I-T^*T)^{1/2}$ is denoted
by $D_T$ and $\sD_T:=\cran D_T$. For defect operators one has the
commutation relations
\[
 TD_T = D_{T^*}T, \quad T^*D_{T^*}=D_{T}T^*.
\]
A linear relation $\cA$ in a Hilbert space $\sH$ is called
symmetric if $\cA\subset \cA^*$, selfadjoint if $\cA=\cA^*$,
skew-symmetric if $\cA\subset -\cA^*$, skew-selfadjoint if $\cA=- \cA^*$,
and nonnegative if $(f',f)\ge 0$ for all $\{f,f'\}\in \cA$. Throughout
this paper we consider separable Hilbert spaces over the field $\dC$
of complex numbers. For general treatments and various standard properties of linear relations
used in this paper we refer to \cite{Ar, codd1974, CS, DdeS}.

\section{Transforms of linear relations, Nevanlinna families and the Schur class}

\subsection{Transforms of linear relations in orthogonally decomposed Hilbert spaces}
Let $\sH$ be a Hilbert space and let $\sM$ be a subspace of $\sH$ and decompose $\sH=\sM\oplus\sK$, where $\sK:= \sH\ominus\sM$.
Define a fundamental symmetry in $\sH=\sM\oplus\sK$ by
\begin{equation}\label{krefund}
\wh J_\sM=\begin{bmatrix} -I_\sM&0\cr 0& I_\sK\end{bmatrix}.
\end{equation}
The adjoint of (the graph of) $T$ w.r.t. to the indefinite inner product $(\wh J_\sM h,k)_\sH$, $h,k\in\sH$,
is denoted by $T^{[*]}:=\wh J_\sM T^*\wh J_\sM$, where $T^*$ stands for the Hilbert space adjoint of $T$ in $\sH$
w.r.t. to the original inner product $(h,k)_\sH$, $h,k\in\sH$. Then one can define the notions of $\wh J_\sM$-symmetric ($\wh B\subset \wh B^{[*]}$),
$\wh J_\sM$-selfadjoint ($\wh B=\wh B^{[*]}$), and $\wh J_\sM$-dissipative ($\IM (\wh J_\sM u',u)\geq 0,\, \{u,u'\}\in\wh B$) for linear relations $\wh B$ in $\sH$.

The main properties of the transformation $\sP_\sM$ in \eqref{ghtjha21} are described in the next proposition.

\begin{proposition}
\label{ghblev}
Let $\wt A$ be a linear relation in the Hilbert space $\sH=\sM\oplus\sK$ and let $\wh B=\sP_\sM(\wt A)$ be defined by \eqref{ghtjha21}.
Then:
\begin{enumerate}
\def\labelenumi{\rm (\roman{enumi})}
\item the transformation $\sP_\sM$ preserves adjoints as follows
\begin{equation}\label{adjoints}
 \sP_\sM(\wt A^{*})=\wh B^{[*]}
\end{equation}
and it establishes a one-to-one correspondence between symmetric (selfadjoint, (maximal) dissipative) relations $\wt A$ in $\sH$ to $\wh J_\sM$-symmetric (resp. $\wh J_\sM$-selfadjoint, (maximal) $\wh J_\sM$-dissipative) relations $\wh B$ in $\sH$;

\item if $\wh h=\{h,h'\},\wh k=\{k,k'\}\in \sH^2$ and $\{u,u'\}=\sP_\sM\wh h, \{v,v'\}=\sP_\sM\wh k$, then
\begin{equation}\label{skew}
 (h',k)+(h,k')=(u',v)+(u,v'),
\end{equation}
in particular, the transformation $\sP_\sM$ preserves the real parts, $\RE (f',f)=\RE (u',u)$ and, hence,
$\wt A$ is accretive (maximal accretive, skew-symmetric, skew-selfadjoint) precisely when the transform
$\wh B=\sP_\sM(\wt A)$ is accretive (resp. maximal accretive, skew-symmetric, skew-selfadjoint);

\item the transformation $\sP_\sM$ establishes a one-to-one correspondence between nonnegative (nonnegative selfadjoint, i.e. maximal nonnegative)
relations $\wt A$ in $\sH$ and $\wh J_\sM$-symmetric accretive (resp. $\wh J_\sM$-selfadjoint maximal accretive) relations $\wh B$ in $\sH$;

\item  the nonnegative selfadjoint relation $\wt A$ and its $\wh J_\sM$-selfadjoint maximal accretive transform $\wh B=\sP_\sM(\wt A)$ are simultaneously $\sM$-minimal,
\begin{equation}\label{ABmin}
 \cspan\left\{\sM+\left(\wt A-\lambda I\right)^{-1}\sM:\lambda\in \dC\setminus\dR_+\right\}
 =\cspan\left\{\sM+\left(\wh B-\lambda I\right)^{-1}\sM: \RE\lambda<0\right\}.
\end{equation}
\end{enumerate}
\end{proposition}

\begin{proof}
(i) Let $h=\begin{bmatrix} \f\cr f \end{bmatrix},\, h'=\begin{bmatrix} \f'\cr f'\end{bmatrix}\in\sH$
and $k=\begin{bmatrix} \psi\cr g \end{bmatrix},\, k'=\begin{bmatrix} \psi'\cr g'\end{bmatrix}\in\sH$
be decomposed according to $\sH=\sM\oplus\sK$. Then
\begin{equation}\label{adj1}
\begin{array}{ll}
 (h',k)-(h,k') &=(\f',\psi)+(f',g)-(\f,\psi')-(f,g') \\
               &=\left(\wh J_\sM \begin{bmatrix} \f \cr f' \end{bmatrix}, \begin{bmatrix} \psi' \cr g \end{bmatrix} \right)
  -\left(\wh J_\sM \begin{bmatrix} \f' \cr f \end{bmatrix},\begin{bmatrix} \psi \cr g' \end{bmatrix} \right).
\end{array}
\end{equation}
By applying this identity to the elements $\wh h=\{h,h'\}\in\wt A$ and $\sP_\sM\wh h\in\sP_\sM(\wt A)=\wh B$ one concludes that
\[
 \wh k=\{k,k'\}\in \wt A^* \quad\Longleftrightarrow\quad
 \sP_\sM\wh k=\left\{\begin{bmatrix} \psi' \cr g \end{bmatrix},\begin{bmatrix} \psi \cr g' \end{bmatrix} \right\} \in \wh B^{[*]}.
\]
This proves \eqref{adjoints}. Hence, in particular,
\[
 \wt A\subset \wt A^* \quad\Longleftrightarrow\quad \wh B\subset  \wh B^{[*]},\quad
 \wt A=\wt A^* \quad\Longleftrightarrow\quad \wh B=\wh B^{[*]}.
\]
Moreover, by applying \eqref{adj1} with $\wh h=\wh k$ shows that
\[
 \IM \left(h',h \right) = \IM \left(\wh J_\sM \begin{bmatrix} \f \cr f' \end{bmatrix}, \begin{bmatrix} \f' \cr f \end{bmatrix} \right)
\]
and hence $\wt A\geq 0$ is (maximal) dissipative precisely when $\wt B$ is (maximal) $\wh J_\sM$-dissipative.

(ii) The formula \eqref{skew} follows from
\[
 (h',k)+(h,k')=(\f',\psi)+(f',g)+(\f,\psi')+(f,g')
 =\left( \begin{bmatrix} \f \cr f' \end{bmatrix}, \begin{bmatrix} \psi' \cr g \end{bmatrix}\right)
  +\left( \begin{bmatrix} \f' \cr f \end{bmatrix}, \begin{bmatrix} \psi \cr g' \end{bmatrix}\right).
\]
With $\wh h=\wh k$ and $\{u,u'\}=\sP_\sM\wh h$ this identity shows that $2\RE (h',h)=2\RE (u',u)$.
Hence $\RE (h',h)\geq 0,\,(=0)$ for all $\wh h\in \wt A$ if and only if $\RE (u',u)\geq 0,\,(=0)$ for all $\wh u\in \wh B$,
which proves the assertions.

(iii) This is obtained by combining the statements in (i) and (ii).

(iv) By item (iii) $\wt A$ is nonnegative and selfadjoint precisely when  $\wh B=\sP_\sM(\wt A)$ is $\wh J_\sM$-selfadjoint and maximal accretive.
Hence, if $\RE \lambda<0$ then $\lambda\in \rho(\wt A)\cap \rho(\wh B)$. Now let $\{h,h'\}\in \wt A$ and decompose $h,h'\in\sH$ as in the proof of (i).
Then, equivalently,
\[
 \left\{\begin{bmatrix} \f'-\lambda \f \cr f'-\lambda f\end{bmatrix},\begin{bmatrix} \f\cr f\end{bmatrix} \right\} \in (\wt A-\lambda)^{-1}
\quad\Longleftrightarrow\quad
 \left\{\begin{bmatrix} \f-\lambda \f' \cr f'-\lambda f\end{bmatrix},\begin{bmatrix} \f'\cr f\end{bmatrix} \right\} \in (\wh B-\lambda)^{-1}.
\]
Since
\[
 \begin{bmatrix} \f'-\lambda \f \cr f'-\lambda f\end{bmatrix} \in \begin{array}{c}\sM\\ \oplus\\ \{0\}\end{array}
\quad\Longleftrightarrow\quad   f'=\lambda f \quad\Longleftrightarrow\quad
 \begin{bmatrix} \f-\lambda \f' \cr f'-\lambda f\end{bmatrix} \in \begin{array}{c}\sM\\ \oplus\\ \{0\}\end{array},
\]
it is seen that for every fixed $\lambda\in\cmr$ with $\RE \lambda<0$ one has
\[
\begin{array}{ll}
\sM+(\wt A-\lambda)^{-1}\sM
 &=\sM+\left\{ f\in \sK: \wh h=\{h,h'\}\in \wt A,\; f'=\lambda f  \right\}\\
 &=\sM+\left\{ f\in \sK: \sP_\sM \wh h\in \sP_\sM(\wt A)=\wh B,\; f'=\lambda f  \right\}\\
 &=\sM+(\wh B-\lambda)^{-1}\sM.
\end{array}
\]
This implies the equality \eqref{ABmin}, i.e., $\wt A$ and $\wh B=\sP_\sM(\wt A)$ are simultaneously $\sM$-minimal.
 \end{proof}

The main properties of the transformation $\sJ_\sM$ in \eqref{ghtjha1} are easier to describe. We state them for a slightly more general
transformation $\sJ_\sM^{(c)}$ involving a unimodular constant $|c|=1$:
\begin{equation}
\label{ghtjha1c}
\sJ_\sM^{(c)}:
 \left\{ \begin{bmatrix} \f \\ f \end{bmatrix},
        \begin{bmatrix} \f' \\ f' \end{bmatrix}
 \right\}
 {\mapsto}
 \left\{ \begin{bmatrix} -c\f'\\ f  \end{bmatrix},
         \begin{bmatrix}  c\f \\ f'\end{bmatrix}
 \right\},\quad   \f,\f'\in\sM,\; f,f'\in\sK,
 \end{equation}
where the choice $c=i$ gives the transformation $\sJ_\sM$ defined in \eqref{ghtjha1};
for simplicity the superscript $(c)$ is dropped in this case.

\begin{proposition}\label{ghblev1}
Let $\wt A$ be a linear relation in the Hilbert space $\sH=\sM\oplus\sK$ and let $\wh A=\sJ_\sM^{(c)}(\wt A)$, $|c|=1$, be defined by \eqref{ghtjha1c}.
Then:
\begin{enumerate}
\def\labelenumi{\rm (\roman{enumi})}
\item the transformation $\sJ_\sM^{(c)}$ satisfies the identity
\begin{equation}\label{adjointc}
 \sJ_\sM^{(c)}(\wt A^{*})=(\wh A)^{*},
\end{equation}
in particular, $\sJ_\sM^{(c)}$ preserves the classes of symmetric, selfadjoint, and (maximal) dissipative relations in $\sH$;

\item  the selfadjoint relations $\wt A$ and $\wh A$ are simultaneously $\sM$-minimal.
\end{enumerate}
\end{proposition}

\begin{proof}
(i) Again let $h=\begin{bmatrix} \f\cr f \end{bmatrix},\, h'=\begin{bmatrix} \f'\cr f'\end{bmatrix}\in\sH$
and $k=\begin{bmatrix} \psi\cr g \end{bmatrix},\, k'=\begin{bmatrix} \psi'\cr g'\end{bmatrix}\in\sH$
be decomposed according to $\sH=\sM\oplus\sK$ and let $\sJ_\sM^{(c)}\wh h$ and $\sJ_\sM^{(c)}\wh k$ be the transforms of $\wh h=\{h,h'\}$ and
$\wh k=\{k,k'\}$, respectively. Then
\[
\begin{array}{ll}
 (h',k)-(h,k') &=(\f',\psi)+(f',g)-(\f,\psi')-(f,g') \\
               &=\left(\begin{bmatrix} c\f \cr f' \end{bmatrix}, \begin{bmatrix} -c\psi' \cr g \end{bmatrix} \right)
  -\left(\begin{bmatrix} -c\f' \cr f \end{bmatrix},\begin{bmatrix} c\psi \cr g' \end{bmatrix} \right).
\end{array}
\]
Therefore, with $\wh h\in\wt A$ and $\sJ_\sM^{(c)}\wh h\in\wh A$ one concludes that
\[
 \wh k\in \wt A^* \quad\Longleftrightarrow\quad
 \sJ_\sM^{(c)}\wh k=\left\{\begin{bmatrix} -c\psi' \cr g \end{bmatrix},\begin{bmatrix} c\psi \cr g' \end{bmatrix} \right\} \in (\wh A)^*.
\]
This proves \eqref{adjointc} and, in particular, one has
\[
 \wt A\subset (\wt A)^* \quad\Longleftrightarrow\quad \wh A\subset  (\wh A)^*,\quad
 \wt A=(\wt A)^* \quad\Longleftrightarrow\quad \wh A=(\wh A)^*.
\]
Moreover, the identity
\[
 \IM \left(h',h \right) = \IM \left(\begin{bmatrix} c\f \cr f' \end{bmatrix}, \begin{bmatrix} -c\f' \cr f \end{bmatrix} \right)
\]
shows that $\wt A$ is (maximal) dissipative precisely when $\wh A$ is (maximal) dissipative.

(ii) Let $\wh h\in \wt A$ and decompose $h,h'\in\sH$ as in the proof of (i).
Then
\[
 \left\{\begin{bmatrix} \f'-\lambda \f \cr f'-\lambda f\end{bmatrix},\begin{bmatrix} \f\cr f\end{bmatrix} \right\} \in (\wt A-\lambda)^{-1}
\quad\Longleftrightarrow\quad
 \left\{\begin{bmatrix} c(\f+\lambda \f') \cr f'-\lambda f\end{bmatrix},\begin{bmatrix} -c\f'\cr f\end{bmatrix} \right\} \in (\wh A-\lambda)^{-1}.
\]
Since
\[
 \begin{bmatrix} \f'-\lambda \f \cr f'-\lambda f\end{bmatrix} \in \begin{array}{c}\sM\\ \oplus\\ \{0\}\end{array}
\quad\Longleftrightarrow\quad   f'=\lambda f \quad\Longleftrightarrow\quad
 \begin{bmatrix} c(\f+\lambda \f') \cr f'-\lambda f\end{bmatrix} \in \begin{array}{c}\sM\\ \oplus\\ \{0\}\end{array},
\]
it is seen that for every fixed $\lambda\in\cmr$,
\[
\begin{array}{ll}
\sM+(\wt A-\lambda)^{-1}\sM
 &=\sM+\left\{ f\in \sK: \wh h\in \wt A,\; f'=\lambda f  \right\}\\
 &=\sM+\left\{ f\in \sK: \sJ_\sM^{(c)} \wh h\in \wh A,\; f'=\lambda f  \right\}\\
 &=\sM+(\wh A-\lambda)^{-1}\sM.
\end{array}
\]
Hence, $\wt A$ and $\wh A$ are simultaneously $\sM$-minimal.
 \end{proof}

\begin{remark}
\label{relvsop}
If $\wt A$ is a nonnegative selfadjoint operator in $\sH=\sM\oplus\sK$ and $\ker \wt A\cap\sM=\{0\}$,
then $\wh A=\sJ_\sM(\wt A)$ and $\wh B=\sP_\sM(\wt A)$ are operators as well and, moreover,
\[
\begin{array}{l}
\dom \wh A=(-iP_\sM \wt A+P_\sK)\dom \wt A,\; \wh A(-iP_\sM\wt Af+P_\sK f)=iP_\sM f+ P_\sK\wt A f,\;f\in\dom \wt A,\\[3mm]
\dom \wh B=(P_\sM \wt A+P_\sK)\dom \wt A,\; \wh B(P_\sM\wt Af+P_\sK f)=P_\sM f+ P_\sK\wt A f,\;f\in\dom \wt A.
\end{array}
\]
\end{remark}

\subsection{Compressed resolvents, Nevanlinna families, and the Schur class}
As indicated in \eqref{opexpr1}, \eqref{opexpr1B} the resolvents and compressed resolvents
are closely related with Nevanlinna families in $\sM$. Some further insight in this connection can be obtained by
connecting selfadjoint relations $\wt A$ with unitary operators $U$ and Nevanlinna families in $\sM$
with operator valued functions from the Schur class $\bS(\sM)$.
In this subsection some basic connections between these objects are recalled and then augmented with some formulas
that will be needed in later sections.

\subsubsection{Cayley transforms}
The basic connection between selfadjoint relations $\wt A$ and unitary operators $U$ is obtained by the direct/inverse Cayley transform:
\[
 \begin{array}{l}
  \wt A\mapsto U=\cC(\wt A):=\left\{\left\{(f'+if,f'-if\right\}:\{f,f'\}\in \wt A\right\}=I-2i(\wt A+iI)^{-1},\\
  U\mapsto \wt A=\cC^{-1}(U):=\left\{\left\{(I-U)g,i(I+U)g\right\}:g\in \sH\right\} =-iI+2i(I-U)^{-1}.
 \end{array}
\]
These formulas establish a one-to-one correspondences between unitary operators $U$ and selfadjoint relations
$\wt A$ in a Hilbert space $\sH$ with $\mul \wt A=\ker(I-U)$.
The resolvents of $\wt A$ and $U$ are connected by the relations
\begin{equation} \label{theres}
\left\{\begin{array}{l} \left(\wt A-\lambda
I\right)^{-1}=-\cfrac{1}{\lambda+i}\, I-\cfrac{2i}{\lambda^2+1}\left(I-\cfrac{\lambda+i}{\lambda-i}\,
U\right)^{-1},\;\lambda\in\rho(\wt A),\;\lambda\ne \pm i,\\[4mm]
\left(\wt A+iI\right)^{-1}=\frac{1}{2i}(I-U), \quad \left(\wt A-iI\right)^{-1}=\frac{1}{2i}(U^{-1}-I),
\end{array}\right. 
\end{equation}
and
\[
(I-z U)^{-1}=\cfrac{1}{1-z}\, I-\cfrac{2i z}{(1-z)^2}\left(\wt A+i\cfrac{1+z}{1-z}\, I\right)^{-1},\;z^{-1}\in\rho(U)\;\mbox{or}\; z=0.
\]

\subsubsection{Connection between the Nevanlinna families and the Schur class}
A relationship between the class $\wt R(\sM)$ of all Nevanlinna families in $\sM$
and the Schur class $\bS(\sM)$ can be given by the linear fractional transformations of functions
and their independent variables, cf. \cite{BHS2008OPu, BHS2009CAOT}:
\begin{equation}\label{dlp1}
\wt R(\sM)\ni\cM\mapsto\Psi(z):=
I+2i\left(\cM\left(i\frac{z+1}{z-1}\right)-iI\right)^{-1}\in\bS(\sM),
\end{equation}
\begin{multline}\label{connectt}
\bS(\sM)\ni\Psi\mapsto \cM(\lambda) \\
=\left\{\begin{array}{l}
\left\{\left\{\left(I-\Psi\left(\cfrac{\lambda+i}{\lambda-i}\right)\right)h,\;
-i\left(I+\Psi\left(\cfrac{\lambda+i}{\lambda-i}\right)\right)h\right\},\;h\in\sM\right\},\;
\IM\lambda<0,\\
\left\{\left\{\left(I-\Psi^*\left(\cfrac{\bar\lambda+i}{\bar\lambda-i}\right)\right)h,\;
i\left(I+\Psi^*\left(\cfrac{\bar\lambda+i}{\bar\lambda-i}\right)\right)h\right\},\;h\in\sM\right\},\;
\IM\lambda>0
\end{array}\right. \\
\quad\in \wt R(\sM).
\end{multline}

\subsubsection{Connections with compressed resolvents}
Let the selfadjoint relation $\wt A$ in $\sH$ and the unitary operator $U$ be connected by the Cayley transform $U=\cC(\wt A)$.
Let $\sM$ is a subspace of $\sH$ and decompose $\sH=\sM\oplus\sK$. Then the connection between the resolvents stated in \eqref{theres}
leads to useful connections between the compressed resolvents and the classes $\wt R(\sM)$ and $\bS(\sM)$.

It follows from the Schur-Frobenius block formula \eqref{Sh-Fr1} for the resolvent $(I-zU)^{-1}$ that
\begin{equation}\label{sh-fros}
P_\sM(I_\sH-z U)^{-1}\uphar\sM=(I_\sM-z\Psi(z))^{-1},\; z\in\dD.
\end{equation}
On the other hand, using the formulas \eqref{theres} and \eqref{dlp1} one gets
\begin{equation}\label{sa-fros}
 P_\sM(\wt A-\lambda I)^{-1}\uphar\sM=-(\cM(\lambda)+\lambda I)^{-1},\; \lambda\in\cmr,
\end{equation}
where $\cM\in\wt R(\sM)$ and $\Psi\in\bS(\sM)$ are connected by \eqref{dlp1}, \eqref{connectt}.

The connection between the resolvents of $\wt A$ and $U=\cC(\wt A)$ in \eqref{theres} implies
that there is also a direct connection between the minimality of $\wt A$ and $U$:
\begin{equation}\label{AUmin}
 \sH=\cspan\left\{\sM+\left(\wt A-\lambda I\right)^{-1}\sM:\lambda\in \dC\setminus\dR\right\}
    =\cspan\left\{(I-\xi U)^{-1}\sM:|\xi|\ne 1\right\},
\end{equation}
i.e., $\wt A$ and $U$ are simultaneously $\sM$-minimal. Since $U$ is unitary, \eqref{AUmin} is equivalent to
\begin{equation}\label{Umin}
\cspan\left\{U^n\sM: n\in\dZ\right\}=\sH.
\end{equation}
If, in addition, $U$ is represented as a $2\times 2$ block operator
 \[
U=\begin{bmatrix} D&C \cr B&F\end{bmatrix}:\begin{array}{l} \sM \\\oplus\\ \sK \end{array} \to
\begin{array}{l} \sM \\\oplus\\ \sK
\end{array},
\]
then the condition \eqref{Umin} can be rewritten equivalently, cf. \cite[Proposition 7.4.]{ArlMFAT2009},
 in one of the following forms:
\begin{multline}\label{smplty}
\cspan\left\{U^n\sM: n\in\dZ\right\}=\sH\Longleftrightarrow\cspan\{F^{n}B\sM; F^{*n}C^*\sM:\,n\in\dN_0\}=\sH\\
\Longleftrightarrow\left(\bigcap\limits_{n=0}^\infty\ker(B^*F^{*n})\right)\cap
\left(\bigcap\limits_{n=0}^\infty\ker(CF^{n})\right)=\{0\}\\
 \iff
\left(\bigcap\limits_{z\in\cU}\ker(B^*(I_\sK -z F^*)^{-1})\right)\cap
\left(\bigcap\limits_{z\in\cU}\ker(C(I_\sK -z F)^{-1})\right)=\{0\},\\
\end{multline}
where $\cU$ can be taken to be some neighborhood of the origin. The first condition involving the block entries of $U$
is often used as definition for $U$ to be simple conservative realization of the function $\Psi(\cdot)\in \bS(\sM)$;
see Appendix \ref{AppendA}.

The next result gives representations for the functions $\cM(\lambda)\in \wt R(\sM)$, $-\cM(\lambda)^{-1}\in \wt R(\sM)$,
and $-\cM(1/\lambda)\in \wt R(\sM)$ as compressed resolvents of certain selfadjoint relations.

\begin{theorem}
\label{inverse1}
Let $\cM(\cdot)$ be a Nevanlinna family in the
Hilbert space $\sM$. Then, up to unitary
equivalence, there exists a unique selfadjoint relation $\wt A$ in the Hilbert space $\sM\oplus\sK$ which is $\sM$-minimal and such that
\begin{enumerate}
  \item the Nevanlinna family $\cM(\lambda)$ has the representation
\begin{equation}
\label{opexpr}
\cM(\lambda)=-\left(P_\sM\left(\wt A-\lambda I\right)^{-1}\uphar\sM\right)^{-1}-\lambda I_\sM ,\quad \lambda\in \cmr.
\end{equation}

  \item  Moreover, if $\wh A=\sJ_\sM(\wt A)$ is as defined in \eqref{ghtjha1}, then
\begin{equation}
\label{opexpr11}
 -\cM^{-1}(\lambda)=-\left(P_\sM\left(\wh A-\lambda I\right)^{-1}\uphar\sM\right)^{-1}-\lambda I_\sM ,\quad \lambda\in\cmr,
\end{equation}

  \item and if $\breve{A} =-\sJ_\sK(\wt A)$ is as defined in \eqref{ghtjha1}, then
\begin{equation}
\label{opexpr112}
 -\cM\left(\cfrac{1}{\lambda}\right)=-\left(P_\sM\left(\breve{A}-\lambda I\right)^{-1}\uphar\sM\right)^{-1}-\lambda I_\sM ,\quad \lambda \in\cmr.
\end{equation}
\end{enumerate}

\end{theorem}
\begin{proof}
As indicated the statement (1) is known, see \cite{DHMS06} (cf. \eqref{opexpr1B}), and an alternative proof
via Cayley transforms is contained in \cite{BHS2008OPu, BHS2009CAOT}.
For later purposes it is convenient to derive the stated representations simultaneously
by connecting each of the functions in (1) -- (3) via Cayley transforms to functions from the Schur class $\bS(\sM)$.

(1) Let $\cM(\lambda)\in \wt R(\sM)$ and define
\[
\Psi(z)=I+2i\left(\cM\left(i\frac{z+1}{z-1}\right)-iI\right)^{-1},\;
z\in\dD.
\]
Then $\Psi$ is an operator valued function that belongs to the Schur class $\bS(\sM)$.
Hence one can represent $\Psi(z)$ as the transfer function of a unique (up to unitary similarity) simple conservative system
\[
\tau=\left\{\begin{bmatrix} U_{11}&U_{12} \cr
 U_{21}&U_{22}\end{bmatrix};\sM,\sM,\sK\right\}
\]
with a state space $\sK$; cf. Appendix \ref{AppendA}. Thus,
\[
 \Psi(z)=U_{11}+zU_{12}(I-zU_{22})^{-1}U_{21}, \quad |z|<1,
\]
where
\[
U=\begin{bmatrix} U_{11}&U_{12} \cr
U_{21}&U_{22}\end{bmatrix}:\begin{array}{l}\sM\\\oplus\\\sK\end{array}\to \begin{array}{l}\sM\\\oplus\\\sK\end{array}
\]
is a unitary operator.
Then the inverse Cayley transform of $U$ given by
\[
\wt A=\left\{\left\{(I-U)g,i(I+U)g\right\}:
g\in\sM\oplus\sK\right\}
\]
is selfadjoint. Using the equivalence of \eqref{sh-fros} and \eqref{sa-fros} the formula \eqref{opexpr} follows; cf. \eqref{opexpr1B}.
The uniqueness property of $\wt A$ holds by the $\sM$-minimality of $\wt A$; see \eqref{AUmin}.

(2) Here the following modification of the unitary block operator $U$ and the system $\tau$ from the proof of (1) are introduced:
\[
\tau'=\left\{U'=\begin{bmatrix} -U_{11}&-U_{12} \cr
 U_{21}&U_{22}\end{bmatrix};\sM,\sM,\sK\right\}.
\]
Clearly, $U'$ is a unitary operator and hence $\tau'$ is a conservative system.
Moreover, it is seen from \eqref{smplty} that $\tau'$ is simple precisely when $\tau$ is simple.
The inverse Cayley transform of $U'$,
\[
 \wh A=\sJ_\sM(\wt A)=\left\{\left\{(I-U')g,i(I+U')g\right\}:\, g\in\sM\oplus\sK\right\},
\]
is selfadjoint and (together with $\wt A$) also $\sM$-minimal; see Proposition \ref{ghblev1}.
Moreover, the transfer function $\Psi_{\tau'}(z)$ is given by $\Psi_{\tau'}(z)=-\Psi(z)$.
By applying the formula \eqref{connectt} to the function $\Psi_{\tau'}(z)$ one obtains for
its transform $\cM'(\lambda)$ the representation
\[
\cM'(\lambda)
 =\left\{\begin{array}{l}
\left\{\left\{\left(I-\Psi_{\tau'}\left(\cfrac{\lambda+i}{\lambda-i}\right)\right)h,\;
-i\left(I+\Psi_{\tau'}\left(\cfrac{\lambda+i}{\lambda-i}\right)\right)h\right\},\;h\in\sM\right\},\;
\IM\lambda<0,\\
\left\{\left\{\left(I-\Psi^*_{\tau'}\left(\cfrac{\bar\lambda+i}{\bar\lambda-i}\right)\right)h,\;
i\left(I+\Psi^*_{\tau'}\left(\cfrac{\bar\lambda+i}{\bar\lambda-i}\right)\right)h\right\},\;h\in\sM\right\},\;
\IM\lambda>0.
\end{array}\right.
\]
Now substitute $\Psi_{\tau'}(z)=-\Psi(z)$ and compare the resulting formula with \eqref{connectt} to see that $\cM'(\lambda)=-\cM(\lambda)^{-1}$.
It remains to replace $\wt A$ by $\wh A$ and $\cM(\lambda)$ by $-\cM(\lambda)^{-1}$ in the representation \eqref{opexpr}
to get \eqref{opexpr11}.

(3) Arguing as above introduce the conservative system
\[
\tau''=\left\{U''=\begin{bmatrix} U_{11}&U_{12} \cr
-U_{21}&-U_{22}\end{bmatrix};\sM,\sM,\sK\right\}
\]
and its selfadjoint transform
\[
\sJ_\sK(\wt A)=\left\{\left\{(I-U'')g,i(I+U'')g\right\}:
g\in\sM\oplus\sK\right\}
\]
which is $\sM$-minimal since $\tau''$ (together with $\tau$) is simple, or equivalently, $\wt A$ is $\sM$-minimal.
The corresponding transfer function $\Psi_{\tau''}(z)$ is given by $\Psi_{\tau''}(z)=\Psi(-z)$, whose
transform $\cM''(z)$ has the expression
\[
\cM''(\lambda)=\left\{\begin{array}{l}
 \left\{\left\{\left(I-\Psi_{\tau''}\left(\cfrac{\lambda+i}{\lambda-i}\right)\right)h,\;
-i\left(I+\Psi_{\tau''}\left(\cfrac{\lambda+i}{\lambda-i}\right)\right)h\right\},\;h\in\sM\right\},\;
\IM\lambda<0,\\
\left\{\left\{\left(I-\Psi^*_{\tau''}\left(\cfrac{\bar\lambda+i}{\bar\lambda-i}\right)\right)h,\;
i\left(I+\Psi^*_{\tau''}\left(\cfrac{\bar\lambda+i}{\bar\lambda-i}\right)\right)h\right\},\;h\in\sM\right\},\;
\IM\lambda>0.
\end{array}\right.
\]
This means that $\cM''(\lambda)=\cM\left(-\cfrac{1}{\lambda}\right)$.
Finally, by replacing $\wt A$ by $\breve{A} =-\sJ_\sK(\wt A)$ and $\cM(\lambda)$ by $-\cM''(-\lambda)$ in
\eqref{opexpr} leads to \eqref{opexpr112}. This completes the proof.
\end{proof}

\section{Representations of Stieltjes and inverse Stieltjes families}\label{sec3}

In Theorem \ref{inverse1} expressions for an arbitrary Nevanlinna family $\cM(\lambda)$ and its transforms $-\cM(\lambda)^{-1}$ and
$-\cM\left(\cfrac{1}{\lambda}\right)$ were given. In this section we assume in addition that $\cM(\lambda)$ is a Stieltjes or an inverse Stieltjes family and construct various representations that take into account the additional properties of $\cM(\lambda)$ implied by these further assumptions.

\subsection{Stieltjes/inverse Stieltjes families and the combined Nevanlinna-Schur class}

It this subsection the classes $\wt\cS(\sM)$ and $\wt\cS^{(-1)}(\sM)$ of Stieltjes and inverse Stieltjes families are connected to
a class of functions that has been studied recently in \cite{ArlHassi_2018}. The definition reads as follows.

\begin{definition}\label{rsm}
Let $\sM$ be a Hilbert space. A $\bB(\sM)$-valued Nevanlinna
function $\Omega$ which is holomorphic on
$\dC\setminus\{(-\infty,-1]\cup[1,+\infty)\}$ is said to belong to
the class $\cRS(\sM)$ if
\[
 -I\le \Omega(x)\le I, \quad x\in (-1,1).
\]
\end{definition}

It has been proved in \cite{ArlHassi_2018} that the class $\cRS(\sM)$ is a subclass of Schur functions $\bS(\sM)$. This means that
the class $\cRS(\sM)$ consists of function that are Nevanlinna functions in $\cmr$ and simultaneously Schur functions on the open unit disk.
This class is called a \emph{combined Nevanlinna-Schur class of $\bB(\sM)$-valued operator functions} and explains the notation $\cRS(\sM)$; $\cR$ standing for $R$-functions (Nevanlinna functions) and $\cS$ for Schur functions.
Some further characterizations for $\cRS(\sM)$ can be found in \cite[Theorem~4.1]{ArlHassi_2018}; see also Appendix \ref{AppendA}.

The next lemma connects the classes $\wt\cS(\sM)$ and $\wt\cS^{(-1)}(\sM)$ to the class $\cRS(\sM)$.
It will be used for some further analysis of Stieljes and inverse Stieljes families
and it offers a tool for establishing some compressed resolvent formulas for these classes.

\begin{lemma}\label{TH1}
Let $\Omega\in\cRS(\sM)$. Then for all $\lambda\in \dC\setminus\dR_+$,
\begin{equation}
\label{formula1}
\begin{array}{rl}
\cQ(\lambda)&=-I+2\left(I_\sM-\Omega\left(\cfrac{1+\lambda}{1-\lambda}\right)\right)^{-1}\\
            &=\left\{\left\{\left(I_\sM-\Omega\left(\cfrac{1+\lambda}{1-\lambda}\right)\right)h,
    \left(I_\sM+\Omega\left(\cfrac{1+\lambda}{1-\lambda}\right)\right)h\right\}:\; h\in\sM\right\}
\end{array}
\end{equation}
is a Stieltjes family and
\begin{equation}
\label{formula2}
\begin{array}{rl}
\cR(\lambda)&=I-2\left(I_\sM+\Omega\left(\cfrac{1+\lambda}{1-\lambda}\right)\right)^{-1}\\
            &=\left\{\left\{\left(I_\sM+\Omega\left(\cfrac{1+\lambda}{1-\lambda}\right)\right)h,
    \left(\Omega\left(\cfrac{1+\lambda}{1-\lambda}\right)-I_\sM\right)h\right\}:\; h\in\sM\right\}
\end{array}
\end{equation}
is an inverse Stieltjes family.

Conversely, if $\cQ(\lambda)$ is the Stieltjes family (resp. $\cR(\lambda)$ is an inverse Stieltjes family) in $\sM$,
then there exists a function $\Omega\in\cRS(\sM)$ such that
\eqref{formula1} (resp. \eqref{formula2}) holds.

Furthermore, the functions $\cQ$ in \eqref{formula1} and $\cR$ in \eqref{formula2} are connected by $\cR=-\cQ^{-1}$
and thus $\cQ\in\wt\cS(\sM)$ if and only if $-\cQ^{-1}\in\wt\cS^{(-1)}(\sM)$.
\end{lemma}

\begin{proof}
Observe the following mapping properties
\begin{equation}\label{NevTrans1}
\lambda\in \dC\setminus\dR_+\iff z:=\cfrac{1+\lambda}{1-\lambda}\in\dC\setminus\{(-\infty,-1]\cup[1,+\infty)\},
\end{equation}
with inverse transform for $\lambda$,
\begin{equation}\label{NevTrans2}
 \lambda=\cfrac{z-1}{z+1},\quad \IM \lambda=\cfrac{2\IM z}{|z+1|^2}.
\end{equation}

Now assume that $\Omega(z)\in \cRS(\sM)$ and let $\cQ(\lambda)$ be given by \eqref{formula1}.
Then $\Omega(z)$ is an operator valued Nevanlinna function with $-I\le \Omega(x)\le I$, $x\in (-1,1)$.
Using \eqref{formula1} one obtains
\[
\left(\cQ(\lambda)+\mu I\right)^{-1}=-\cfrac{1}{1-\mu}\left(I-\cfrac{2}{1-\mu}\left(\cfrac{1+\mu}{1-\mu} I+\Omega\left(\cfrac{1+\lambda}{1-\lambda}\right)\right)^{-1}\right).
\]
This shows that $\left(\cQ(\lambda)+\mu I\right)^{-1}$ admits an analytic continuation to the negative semi-axis $(-\infty,0)$.
On the other hand,
\begin{equation}\label{ineqOmega}
 -I\le \Omega(x)\le I \quad\Longleftrightarrow\quad 2(I-\Omega(x))^{-1}\geq I, \quad x\in (-1,1),
\end{equation}
and hence $\cQ(\lambda)\geq 0$ for $\lambda<0$. By Definition \ref{invStieltjes} one concludes that
$\cQ(\lambda)\in\wt\cS(\sM)$.

By comparing the formulas \eqref{formula1} and  \eqref{formula2} it is seen that $\cR(\lambda)=-\cQ^{-1}(\lambda)$.
Therefore, $\cR(\lambda)\in\wt\cS^{(-1)}(\sM)$.

Conversely, assume that $\cQ(\lambda)$ is a Stieltjes family. Since
$\cQ(x)$ is a nonnegative selfadjoint relation for $x<0$, the resolvent
$(\cQ(x)+I)^{-1}:\sM\to \sM$ is bounded for $x<0$. By assumption $\cQ(z)$ is also a Nevanlinna family which admits an analytic
continuation to the semi-axis $(-\infty,0)$ in the resolvent sense; see Definition \ref{invStieltjes}.
The formula
\[
 \left(I+(\nu-\mu)(\cQ(\lambda)-\nu)^{-1}\right)^{-1}=I+(\nu-\mu)(\cQ(\lambda)-\mu)^{-1}
\]
applied to $\mu=-1$ and $|\nu-\mu|<\delta$ (with $\delta$ small enough)
implies that $(\cQ(\lambda)+I)^{-1}$ also admits an analytic continuation
to the semi-axis $(-\infty,0)$, so that $(\cQ(\lambda)+I)^{-1}$ is holomorphic on $\dC\setminus \dR_+$.
In particular, $(\cQ(\lambda)+I)^{-1}$ is bounded, when $\lambda=x+i y$ is sufficiently close to a real point $x<0$.
Since $-(\cQ(\lambda)+I)^{-1}$ is a Nevanlinna family, boundedness at a single point $\lambda_0\in\cmr$ implies
boundedness of $-(\cQ(\lambda)+I)^{-1}$ at every point $\lambda\in\cmr$; see e.g. \cite[Proposition~4.18]{DHMS06}.
Now define
\[
\Omega(z):=I-2\left(I+\cQ\left(\cfrac{z-1}{1+z}\right)\right)^{-1},\;
z\in\dC\setminus\left((-\infty,-1]\cup[1,+\infty)\right).
\]
Then $\Omega(z)$ is a $\bB(\sM)$-valued Nevanlinna
function defined on
$\dC\setminus\left\{(-\infty,-1]\cup[1,+\infty)\right\}$ and
\[
-I\le \Omega(x)\le I,\quad x\in(-1,1),
\]
see \eqref{NevTrans1}, \eqref{NevTrans2}, and \eqref{ineqOmega}.
Hence, $\Omega\in\cRS(\sM)$ and \eqref{formula1} holds.

In the case where $\cR(\lambda)$ is from the inverse Stieltjes class one can use
the result just proved for the Stieltjes family by employing the identity $\cR(\lambda)=-\cQ^{-1}(\lambda)$,
which is clear from \eqref{formula1}, \eqref{formula2}.
\end{proof}

\subsection{Representations by means of compressed resolvents}

In this subsection representation theorems for general Stieltjes or an inverse Stieltjes families are established as compressed resolvents
along the lines of Theorem~\ref{inverse1}. These involve again transformations of a selfadjoint relation $\wt A$ which, in addition, is nonnegative.
In this case it is convenient to introduce the following linear fractional transformation of $\wt A$,
\begin{equation}\label{ATtrans}
 T=-I+2(I+\wt A)^{-1}.
\end{equation}

We start with a lemma containing some simple, but useful, observations.

\begin{lemma}\label{TTlemma}
Let $\wt A$ be a nonnegative selfadjoint relation in the Hilbert space $\sH=\sM\oplus\sK$ and let
$\wh B=\sP_\sM(\wt A)$ be defined by \eqref{ghtjha21}. Then the transforms
\[
 T=-I+2(I+\wt A)^{-1}, \qquad  \wh T=-I+2(I+\wh B)^{-1}
\]
are contractive, $T$ is selfadjoint, $\wh T$ is $\wh J_\sM$-selfadjoint, i.e. $\wh J_\sM\wh T=\wh T^*\wh J_\sM$,
where the fundamental symmetry $\wh J_\sM$ is defined by \eqref{krefund}, and they have block representations
\begin{equation}\label{blockT}
 T=\begin{bmatrix} D&C\cr C^*&F\end{bmatrix}:\begin{array}{l}\sM\\\oplus\\\sK\end{array}\to \begin{array}{l}\sM\\\oplus\\\sK\end{array}
\end{equation}
and
\begin{equation}\label{BTtrans2}
\wh T=\wh J_\sM T=\begin{bmatrix} -D&-C\cr C^*&F\end{bmatrix}:\begin{array}{l}\sM\\\oplus\\\sK\end{array}\to \begin{array}{l}\sM\\\oplus\\\sK\end{array}.
\end{equation}
Conversely, if $T$ is a selfadjoint contraction as in \eqref{blockT} and $\wh T$ is given by \eqref{BTtrans2}, then
\begin{equation}\label{Atrans}
\wt A=\left\{\left\{(I+ T)h,(I-T)h\right\}:\; h\in \sH\right\}
\end{equation}
is a nonnegative selfadjoint relation in $\sH$ and
\begin{equation}\label{BTtrans}
 \wh B=\left\{\left\{(I+\wh T)h,(I-\wh T)h\right\}:\;h\in \sH\right\}
\end{equation}
is maximal accretive $\wh J_\sM$-selfadjoint relation in $\sH$ and, moreover, $\wh B=\sP_\sM(\wt A)$.
\end{lemma}

\begin{proof}
Since the transformation \eqref{ATtrans} is an involution, $T=-I+2(I+\wt A)^{-1}$ and $\wt A$ are connected also by \eqref{Atrans},
and $\wh T=-I+2(I+\wh B)^{-1}$ and $\wh B$ are connected by \eqref{BTtrans}.
It is well known and easy to check that $\wt A$ is selfadjoint and nonnegative precisely when $T$ is a selfadjoint contraction.
Moreover, $\wh B$ is maximal accretive if and only if $\wh T$ is a contraction in $\bB(\sH)$.
On the other hand, by Proposition \ref{ghblev} $\wt A$ is selfadjoint and nonnegative if and only if $\wh B=\sP_\sM(\wt A)$
is $\wh J_\sM$-selfadjoint and maximal accretive.
When $\wh B$ is $\wh J_\sM$-selfadjoint then also $\wh T$ is $\wh J_\sM$-selfadjoint, and conversely.

Now, using \eqref{blockT} and \eqref{Atrans} one gets
\[
\wt A=\left\{\left\{\begin{bmatrix}(I_\sM+D)\f+Cf\cr
C^*\f+(I+F)f \end{bmatrix}, \begin{bmatrix}(I_\sM-D)\f-Cf\cr
-C^*\f+(I-F)f \end{bmatrix}\right\}:\;\f\in\sM,\;f\in\sK\right\},
\]
while \eqref{BTtrans2} and \eqref{BTtrans} lead to
\[
\wh B=\left\{\left\{\begin{bmatrix}(I_\sM-D)\f-Cf\cr
C^*\f+(I+F)f \end{bmatrix}, \begin{bmatrix}(I_\sM+D)\f+Cf\cr
-C^*\f+(I-F)f \end{bmatrix}\right\}:\;\f\in\sM,\;f\in\sK\right\}.
\]
One concludes that $\wh B$ and $\wt A$ are connected by $\wh B=\sP_\sM(\wh A)$ and, conversely,
if $\wh B=\sP_\sM(\wh A)$, then their transforms $\wh T$ and $T$ are connected by $\wh T=\wh J_\sM T$
as in \eqref{BTtrans2}.
\end{proof}

The contractive operators $T$ and $\wh T$ generate discrete-time passive linear systems which are shortly treated in Appendix \ref{AppendA}.
The next lemma connects $\sM$-minimality of $\wt A$ and $\wh B$ in Lemma \ref{TTlemma} to their simplicity.

\begin{lemma}\label{TTlemma2}
Let $\wt A$, $\wh B$, $T$ and $\wt T$ be as in Lemma \ref{TTlemma}.
and let $\tau=\left\{ T,\sM,\sM,\sK\right\}$ and $\tau_-=\left\{ \wh T,\sM,\sM,\sK\right\}$
be the discrete-time passive linear systems generated by $T$ and $\wh T$, respectively (see Appendix \ref{AppendA}).
Then the following statements are equivalent:
\begin{enumerate}
\def\labelenumi{\rm (\roman{enumi})}
\item $\wt A$ is $\sM$-minimal;
\item $\wh B$ is $\sM$-minimal;
\item the system $\tau$ is simple;
\item the system $\tau_-$ is simple.
\end{enumerate}
\end{lemma}

\begin{proof}
The equivalence of (i) and (ii) was proved in Proposition \ref{ghblev} (iv).
To prove their equivalence to (iii) and (iv) decompose $T$ and $\wh T$ as in \eqref{blockT} and \eqref{BTtrans2}.
Then using \eqref{ATtrans} it can be verified that the resolvents of $T$ and $\wt A$ are connected by
\begin{equation}
\label{connectresab}
(\wt A-\lambda
I)^{-1}=\cfrac{1}{1-\lambda}(T+I)\left(I-\frac{1+\lambda}{1-\lambda}T\right)^{-1},\;\lambda\in\dC\setminus\dR_+.
\end{equation}
Similarly the resolvents of $\wh T$ and $\wh B$ are connected by
\begin{equation}\label{connectresab11}
(\wh B -\lambda
I)^{-1}=\cfrac{1}{1-\lambda}(\wh T+I)\left(I-\frac{1+\lambda}{1-\lambda}\wh T\right)^{-1},\;\RE\lambda<0.
\end{equation}
These relations combined with \eqref{SPAN11} in Appendix \ref{AppendA} imply the equalities
\[
\begin{array}{l}
\cspan\left\{\sM+\left(\wt A-\lambda I\right)^{-1}\sM:\lambda\in \dC\setminus\dR_+\right\}
=\cspan\left\{(I-z T)^{-1}\sM:z\in\dD\right\}\\
=\cspan\left\{T^n\sM: n\in\dN_0\right\}=\sM\oplus\cspan\left\{F^nC^*\sM: n\in\dN_0\right\},
\end{array}
\]
and
\[
\begin{array}{l}
\cspan\left\{\sM+\left(\wh B-\lambda I\right)^{-1}\sM: \RE\lambda<0\right\}=\cspan\left\{(I-z \wh T)^{-1}\sM:z\in\dD\right\}\\
=\cspan\left\{\wh T^n\sM: n\in\dN_0\right\}=\sM\oplus\cspan\left\{F^nC^*\sM: n\in\dN_0\right\}.
\end{array}
\]
Since $T=T^*$ and $\wt T=\wh J_\sM T$ has the simple expression \eqref{BTtrans2}
the equality \eqref{SPAN22} involving $T^*$ and $\wh T^*$ yields the same identities as stated above.
Therefore,
$\wh B$ is $\sM$-minimal precisely when $\tau_-$ is simple $\wt A$ is $\sM$-minimal precisely when $\tau$ is simple and
$\wh B$ is $\sM$-minimal precisely when $\tau_-$ is simple.
This proves the remaining equivalences.
\end{proof}

We now consider compressed resolvents of the linear relations $\wt A$ and $\wh B$ appearing in Lemma \ref{TTlemma}.
While the discussion given in the introduction after the formulas \eqref{opexpr1} and \eqref{opexpr1B}
yields the first statement in the next theorem, it is convenient to use here the transformation \eqref{ATtrans} to
keep connections visible with other forthcoming statements.

\begin{theorem}\label{TH0}
Let $\sM$ be a subspace of the Hilbert space $\sH$ and let $\wt A$ is a nonnegative selfadjoint relation in $\sH=\sM\oplus\sK$.
Then the following assertions hold:
\begin{enumerate}
  \item  The compressed resolvent
$P_\sM(\wt A-\lambda I)^{-1}\uphar\sM$ admits the representation
\begin{equation}\label{aarep}
P_\sM(\wt A-\lambda I)^{-1}\uphar\sM=-(\cR(\lambda)+\lambda I_\sM)^{-1},\quad \lambda\in\dC\setminus\dR_+,
\end{equation}
with $\cR\in\wt\cS^{(-1)}(\sM)$.

  \item If $\wh B$ is a maximal accretive and $\wh J_\sM$-selfadjoint relation w.r.t. the fundamental symmetry $\wh J_\sM$ defined in \eqref{krefund}, then the compressed resolvent $P_\sM(\wh B-\lambda I)^{-1}\uphar\sM$ admits the representation
\begin{equation}\label{brep}
P_\sM(\wh B-\lambda I)^{-1}\uphar\sM=(\cQ(\lambda)-\lambda I_\sM)^{-1}, \quad \RE \lambda<0,
\end{equation}
with $\cQ\in\wt\cS(\sM)$.

\item If $\wh A$ is defined as follows (see \eqref{ghtjha21}, \eqref{ghtjha1})
\begin{equation}
\label{ghtjha112}
\wh A=\sJ_\sM\sP_\sM(\wh B)=\left\{\,
 \left\{ \begin{bmatrix}-ih \cr f \end{bmatrix},
        \begin{bmatrix} ih' \cr f' \end{bmatrix} \right\}: \left\{\,
  \begin{bmatrix} h\\ f  \end{bmatrix},
         \begin{bmatrix}  h' \\ f'\end{bmatrix}
          \right\}\in\wh B\right\},
 \end{equation}
then $\wh A$ is a selfadjoint relation and
\begin{equation}\label{arep}
P_\sM(\wh A-\lambda I)^{-1}\uphar\sM=-(\cQ(\lambda)+\lambda I_\sM)^{-1}, \quad \lambda\in\cmr,
\end{equation}
where $\cQ\in\wt\cS(\sM)$ is the same function as in \eqref{brep}.

 \item If $\breve{A} =-\sJ_\sK(\wh A)$, where $\wh A$ is as defined in \eqref{ghtjha112}, then $\breve{A}$ is a selfadjoint nonnegative relation in $\sH$ and
\begin{equation}
\label{opexpr3}
 P_\sM\left(\breve{A}-\lambda I\right)^{-1}\uphar\sM=-\left(-\cQ\left(\cfrac{1}{\lambda}\right)+\lambda I_\sM\right)^{-1},\quad \lambda\in\dC\setminus\dR_+,
\end{equation}
where $-\cQ\left(\cfrac{1}{\lambda}\right)\in\wt\cS^{(-1)}(\sM)$ and $\cQ\in\wt\cS(\sM)$ is the same function as in \eqref{brep}.
\end{enumerate}
Moreover, if $\wt A$ in \eqref{aarep} and $\wh B$ in \eqref{brep} are connected by $\wh B=\sP_\sM(\wt A)$ then $\cQ(\lambda)=-\cR^{-1}(\lambda)$ and, furthermore, $\breve{A}=\wt A^{-1}$.
\end{theorem}

\begin{proof}
(1) Since $\wt A$ is selfadjoint and nonnegative the transform $T=-I+2(I+\wt A)^{-1}$ is a selfadjoint contraction. Decompose $T$ as in \eqref{blockT}.
Then the Schur-Frobenius formula \eqref{Sh-Fr1} for the resolvent of $T$ shows that, cf. \eqref{sch-fr},
\begin{equation}\label{sch-fr1}
P_\sM(I-zT)^{-1}\uphar\sM=(I_\sM-z\Omega(z))^{-1},\quad z\in\dC\setminus\left\{(-\infty,-1]\cup[1,+\infty)\right\},
\end{equation}
where
\[
\Omega(z)=D+zC(I-zF)^{-1}C^*,\quad z\in\dC\setminus\left\{(-\infty,-1]\cup[1,+\infty)\right\}.
\]
According to \cite{ArlHassi_2018} the function $\Omega(z)$ belongs to the class $\cRS(\sM)$;
Thus, by Lemma \ref{TH1} the transform
\[
\cR(\lambda)=I-2\left(I_\sM+\Omega\left(\cfrac{1+\lambda}{1-\lambda}\right)\right)^{-1},\quad \lambda\in \dC\setminus\dR_+,
\]
is an inverse Stieltjes family. Using \eqref{formula2} it is easy to check that
\begin{equation}\label{aarep2}
\left(\cR(\lambda)+\lambda I_\sM\right)^{-1}
=\cfrac{1}{\lambda-1}\left(I_\sM+\Omega\left(\frac{1+\lambda}{1-\lambda}\right)\right)
\left(I_\sM-\frac{1+\lambda}{1-\lambda}\Omega\left(\frac{1+\lambda}{1-\lambda}\right)\right)^{-1}.
 \end{equation}
On the other hand, if follows from \eqref{connectresab} that
\begin{multline}\label{compA2}
P_\sM(\wt A-\lambda
I)^{-1}\uphar\sM=\cfrac{1}{1-\lambda}P_\sM(T+I)\left(I-\frac{1+\lambda}{1-\lambda}T\right)^{-1}\uphar\sM\\
=\cfrac{1}{1-\lambda}\left(I_\sM+\Omega\left(\frac{1+\lambda}{1-\lambda}\right)\right)
\left(I_\sM-\frac{1+\lambda}{1-\lambda}\Omega\left(\frac{1+\lambda}{1-\lambda}\right)\right)^{-1},\quad \lambda\in\dC\setminus\dR_+.
\end{multline}
A comparison with \eqref{aarep2} gives the stated formula \eqref{aarep}.

(2) Consider the transform \eqref{ATtrans} of $\wh B$, $\wh T:=-I+2(I+\wh B)^{-1}$.
By Lemma \ref{TTlemma} $\wh T$ is contractive and $\wh J_\sM$-selfadjoint and
the operator $T:=\wh J_\sM \wh T$ is a selfadjoint contraction of the form \eqref{blockT},
while $\wh T=\wh J_\sM T$ has the form \eqref{BTtrans2}.
Define
\begin{equation}\label{transfer2}
 \Omega(z)=D+zC(I-zF)^{-1}C^*,\qquad \Psi(z)=-\Omega(z),\quad z\in\dC\setminus\left\{(-\infty,-1]\cup[1,+\infty)\right\}.
\end{equation}
Notice that $\Omega\in\cRS(\sM)$ and $\Psi$ are the transfer functions of the discrete-time passive linear systems
$\tau=\left\{ T,\sM,\sM,\sK\right\}$ and $\tau_-=\left\{ \wh T,\sM,\sM,\sK\right\}$ in Lemma \ref{TTlemma2} (see Appendix \ref{AppendA}).
Again by the Schur-Frobenius formula \eqref{Sh-Fr1} we have
\[
 P_\sM(I-z\wh T)^{-1}\uphar\sM=(I_\sM-z\Psi(z))^{-1}=(I_\sM+z\Omega(z))^{-1},\quad z\in\dD.
\]
On the other hand, the resolvent formula \eqref{connectresab11} implies that for all $\RE \lambda<0$ (cf. \eqref{compA2})
\begin{equation}\label{compA3}
\begin{array}{rl}
P_\sM(\wh B -\lambda I)^{-1}\uphar\sM
 &=\cfrac{1}{1-\lambda}P_\sM(\wh T+I)\left(I-\frac{1+\lambda}{1-\lambda}\wh T\right)^{-1}\uphar\sM\\
 &=\cfrac{1}{1-\lambda}\left(I_\sM+\Psi\left(\frac{1+\lambda}{1-\lambda}\right)\right)
    \left(I_\sM-\frac{1+\lambda}{1-\lambda}\Psi\left(\frac{1+\lambda}{1-\lambda}\right)\right)^{-1} \\
 &=\cfrac{1}{1-\lambda}\left(I_\sM-\Omega\left(\frac{1+\lambda}{1-\lambda}\right)\right)
    \left(I_\sM+\frac{1+\lambda}{1-\lambda}\Omega\left(\frac{1+\lambda}{1-\lambda}\right)\right)^{-1}.
\end{array}
\end{equation}
By Lemma \ref{TH1} the function $\cQ$ defined by
\[
\cQ(\lambda)=\left\{\left\{\left(I_\sM-\Omega\left(\cfrac{1+\lambda}{1-\lambda}\right)\right)h,
\left(I_\sM+\Omega\left(\cfrac{1+\lambda}{1-\lambda}\right)\right)h\right\}:\;
h\in\sM\right\},\quad\lambda\in \dC\setminus\dR_+,
\]
belongs to the Stieltjes class $\cS(\sM)$ and for all $\lambda\in\dC\setminus\dR_+$ one obtains
\[
\left(\cQ(\lambda)-\lambda I_\sM\right)^{-1}=\cfrac{1}{1-\lambda}\left(I_\sM-\Omega\left(\frac{1+\lambda}{1-\lambda}\right)\right)
\left(I_\sM+\frac{1+\lambda}{1-\lambda}\Omega\left(\frac{1+\lambda}{1-\lambda}\right)\right)^{-1}.
\]
A comparison with \eqref{compA3} leads to \eqref{brep}.

(3) Let $\wh B$, $\wh T$ and $T$ be as in the proof of item (3) and define
\[
\wt A=\left\{\left\{(I+ T)h,(I-T)h\right\}:\; h\in \sH\right\}.
\]
Then by Lemma~\ref{TTlemma} $\wt A$ is a nonnegative selfadjoint relation, which is connected to $\wh B$ by $\wh B=\sP_\sM(\wt A)$,
i.e., $\wt A=\sP_\sM(\wh B)$ and $\wh A=\sJ_\sM(\wt A)$.
As was proved in (1) the compressed resolvent $P_\sM(\wt A-\lambda I)^{-1}\uphar\sM$ admits the representation \eqref{aarep},
where the function $\cR$ is given by \eqref{formula2} in Lemma \ref{TH1}. On the other hand, the proof of (2) shows that
$P_\sM(\wh B-\lambda I)^{-1}\uphar\sM$ admits the representation \eqref{brep}, where the function $\cQ$ is given by \eqref{formula1}
in Lemma \ref{TH1}. Hence, $\cQ(\lambda)=-\cR^{-1}(\lambda)$ for all $\lambda\in \dC\setminus\dR_+$.
Since $\wh A=\sJ_\sM(\wt A)$ it follows from Theorem \ref{inverse1} that
\[
(\cQ(\lambda)+\lambda I_\sM)^{-1}=-P_\sM(\sJ_\sM(\wt A)-\lambda I)^{-1}\uphar\sM,\quad \lambda\in\cmr.
\]

(4) By assumption $\breve{A} =-\sJ_\sK(\wh A)$ and the proof of (3) shows that $\wh A=\sJ_\sM(\wt A)$, where $\wt A=\sP_\sM(\wh B)$
is a nonnegative selfadjoint relation in $\sH$. Thus,
\[
 \breve{A} =-\sJ_\sK(\wh A)=-\sJ_\sK\sJ_\sM(\wt A)=\wt A^{-1},
\]
and, in particular, $\breve{A}$ is a nonnegative selfadjoint relation in $\sH$.
Moreover, the formula \eqref{opexpr3} follows from Theorem \ref{inverse1} and \eqref{arep}.

The last assertion is clear from the arguments used above to prove (3) and (4).
\end{proof}

The next theorem shows that all inverse Stieltjes families $\cR\in\wt\cS^{(-1)}(\sM)$ can be characterized by the statement (1) in Theorem \ref{TH0}.

\begin{theorem}
\label{obratno1}
Let $\cR$ belong to the inverse
Stieltjes class in $\sM$. Then there exists up to unitary
equivalence a unique nonnegative selfadjoint relation $\wt A$ in the Hilbert space
$\sM\oplus\sK$ such that $\wt A$ is $\sM$-minimal and the relation
\begin{equation}
\label{opexprstiel}
P_\sM\left(\wt A-\lambda I\right)^{-1}\uphar\sM=-\left(\cR(\lambda)+\lambda I_\sM\right)^{-1},\quad \lambda\in\dC\setminus\dR_+
\end{equation}
holds.
\end{theorem}

\begin{proof}
Let $\cR\in \wt\cS^{(-1)}(\sM)$. Then according to Lemma \ref{TH1} the operator valued function
\[
\Omega(z):=-I+2\left(I_\sM-\cR\left(\cfrac{z-1}{1+z}\right)\right)^{-1},
\quad z\in\dC\setminus\left\{(-\infty,-1]\cup[1,+\infty)\right\},
\]
is from the class $\cRS(\sM)$. By \cite[Theorem 4.3]{ArlHassi_2015}
there exists up to the unitary equivalence a unique simple passive
selfadjoint system
$\tau=\left\{ T,\sM,\sM,\sK\right\}$, where $T$ is a selfadjoint contraction,
such that the transfer function of $\tau$ coincides with $\Omega(z)$.
Then the linear fractional transformation
\[
\wt A=-I+2(I+T)^{-1}=\left\{\{(I+T)h,(I-T)h\}:\; h\in\sM\oplus\sK\right\}
\]
is a nonnegative selfadjoint relation in $\sM\oplus\sK$.
Since the system $\tau$ is simple, $\wt A$ is $\sM$-minimal by Lemma \ref{TTlemma2}.
Now it is clear that the formulas \eqref{sch-fr1}, \eqref{aarep2}, and \eqref{compA2} hold
and thus the formula \eqref{opexprstiel} is obtained from Theorem \ref{TH0}.
For the uniqueness of $\wt A$ see the discussion given in Introduction.
\end{proof}

For Stieltjes families $\cR\in\wt\cS(\sM)$ we have the following characterizations.

\begin{theorem}\label{stieltacc}
Let $\cQ$ belong to the Stieltjes class in $\sM$. Then the following statements hold:
 \begin{enumerate}
 \item There exists up to the unitary equivalence a unique maximal accretive linear relation $\wh B$ in the Hilbert space $\sM\oplus\sK$ such that $\wh B$ is $\wh J_\sM$-selfadjoint w.r.t. the fundamental symmetry $\wh J_\sM$ in \eqref{krefund}, such that $\wh B$ is $\sM$-minimal and the relation
\begin{equation}
\label{opexprstielacc}
 P_\sM\left(\wh B-\lambda I\right)^{-1}\uphar\sM=\left(\cQ(\lambda)-\lambda I_\sM\right)^{-1}
 \end{equation}
holds for all $\RE\lambda<0$.
\item There exists up to the unitary equivalence a unique selfadjoint relation $\wh A$ in $\sM\oplus\sK$, such that $\wh A$ is $\sM$-minimal and its transform $\sJ_\sM(\wh A)$ is nonnegative, and the relation
\[
P_\sM\left(\wh A-\lambda I\right)^{-1}\uphar\sM=-\left(\cQ(\lambda)+\lambda I_\sM\right)^{-1}
\]
holds for all $\lambda\in\cmr$.
\end{enumerate}
Furthermore, one can choose
\[
\wh A=\sJ_\sM\sP_\sM(\wh B).
\]
\end{theorem}

\begin{proof}
Let $\cR\in \wt\cS(\sM)$. Then according to Lemma \ref{TH1} the operator valued function
\[
\Omega(z):=I-2\left(I_\sM+\cQ\left(\cfrac{z-1}{1+z}\right)\right)^{-1},\quad z\in\dC\setminus\left\{(-\infty,-1]\cup[1,+\infty)\right\},
\]
belongs to the class $\cRS(\sM)$. Again by \cite[Theorem 4.3]{ArlHassi_2015}
there exists up to the unitary equivalence a unique simple passive
selfadjoint system
$\tau=\left\{ T,\sM,\sM,\sK\right\}$, whose transfer function coincides with $\Omega(z)$.
Here $T$ is a selfadjoint contraction of the form \eqref{blockT} and the associated operator
$\wh T:=\wh J_\sM T$ is a $\wh J_\sM$-selfadjoint contraction in $\sH$ of the form \eqref{BTtrans2}.
The corresponding discrete-time linear system $\tau_-$ appearing in Lemma \ref{TTlemma2}
is also simple and it has the transfer function
\[
 \Psi_{\tau_-}(z)=-\Omega(z),\quad z\in \dC\setminus\left\{(-\infty,-1]\cup[1,+\infty)\right\},
\]
cf. \eqref{transfer2}. By Lemma \ref{TTlemma} the transform $\wt A$ of $T$ defined by \eqref{Atrans} is nonnegative and selfadjoint
and the transform $\wh B$ defined by \eqref{BTtrans} is maximal accretive and $\wh J_\sM$-selfadjoint, and
$\wt A$ and $\wh B$ are connected by $\wh B=\sP_\sM(\wt A)$.
Moreover, by Lemma \ref{TTlemma2} $\wt A$ and $\wh B$ are $\sM$-minimal.
Since we are now in the setting used to prove part (2) of Theorem \ref{TH0} the representation
\eqref{opexprstielacc} is obtained from \eqref{brep}.
The stated uniqueness property of $\wh B$ is a consequence of its $\sM$-minimality.
This completes the proof of the first statement (1).

To prove the second statement (2) consider the transform $\wh A=\sJ_ \sM(\wt A)=\sJ_ \sM\sP_ \sM(\wh B)$.
Then we may apply part (3) of Theorem \ref{TH0}; the stated resolvent formula is immediate from \eqref{arep}.
According to Proposition \ref{ghblev1} $\wh A$ is $\sM$-minimal, since $\wt A$ is $\sM$-minimal,
and this implies the uniqueness property of $\wh A$.
\end{proof}

Next some further representations for Stieltjes and inverse Stieltjes families will be
established by means of some specific transformation properties that these families obey.
The basic properties of scalar Stieltjes and inverse Stieltjes functions can be found in \cite{KacK}.
The next lemma is an extension of \cite[Lemma~S1.5.2, Theorem~S1.5.3]{KacK}
from the scalar case to our present general setting.

\begin{lemma}\label{invthm}
With $\lambda\in \dC\setminus \dR_+$ the following assertions are equivalent:
\begin{enumerate}
\def\labelenumi{\rm (\roman{enumi})}
\item $\cQ(\lambda)\in \wt\cS(\sM)$;
\item $-\cQ^{-1}(\lambda) \in \wt\cS^{(-1)}(\sM)$;
\item $\lambda\cQ(\lambda)\in \wt\cS^{(-1)}(\sM)$.
\end{enumerate}
\end{lemma}
\begin{proof}
The equivalence of (i) and (ii) was proved in Lemma \ref{TH1} (see also Introduction).
We use the same approach to prove the equivalence of (i) and (iii).

Let the functions $\cQ\in\wt \cS(\sM)$ and $\Omega\in\cRS(\sM)$
be connected by \eqref{formula1}. Then a straightforward calculation shows that
\[
 \lambda\cQ(\lambda)=\left\{\left\{\left(I+\Upsilon\left(\cfrac{1+\lambda}{1-\lambda}\right)\right)h,
\left(\Upsilon\left(\cfrac{1+\lambda}{1-\lambda}\right)-I\right)h\right\}:\, h\in\sM\right\},
\quad \lambda\in \dC\setminus\dR_+,
\]
where $\Upsilon(z)$ is given by the linear fractional transformation
\[
 \Upsilon(z)=\left(zI-\Omega(z)\right)\left(I-z\Omega(z)\right)^{-1},\quad
 z\in\dC\setminus\{(-\infty,-1]\cup[1,+\infty)\}.
\]
As shown in \cite[Theorem~4.1]{ArlHassi_2018} the function $\Upsilon(z)$ (together with $\Omega(z)$)
belongs to the class $\cRS(\sM)$; see also \eqref{formula3} in Appendix \ref{AppendA}.
Therefore, by Lemma~\ref{TH1} $\lambda\cQ(\lambda)$ is an inverse Stieltjes family.
Hence, (i)$\Longrightarrow$(iii).

Conversely, let $\lambda\cQ(\lambda)\in \wt \cS^{(-1)}(\sM)$. By applying the proven implications we get
\[
\begin{array}{rl}
 & -\left(\lambda\cQ(\lambda)\right)^{-1} = -\cfrac{\cQ^{-1}(\lambda)}{\lambda}\in \wt \cS(\sM)   \\
 & \Longrightarrow \quad \lambda \left(-\lambda\cQ(\lambda)\right)^{-1}=-\cQ^{-1}(\lambda)\in \wt \cS^{(-1)}(\sM) \\
 & \Longrightarrow \quad \cQ(\lambda)\in \wt\cS(\sM),
\end{array}
\]
which proves the implication (iii)$\Longrightarrow$(i).
\end{proof}

Notice that Lemma \ref{invthm} gives also the equivalence
\[
 \cR(\lambda)\in\wt\cS^{(-1)}(\sM) \iff \frac{\cR(\lambda)}{\lambda} \in\wt\cS(\sM).
\]
One concludes that every function $\cR(\lambda)\in \wt\cS^{(-1)}(\sM)$ is of the form $\cR(\lambda)=\lambda\cQ(\lambda)$ for some
$\cQ(\lambda)\in \wt\cS(\sM)$. Similarly every function
$\cQ(\lambda)\in \wt\cS(\sM)$ is of the form $\cQ(\lambda)=\frac{\cR(\lambda)}{\lambda}$ for some
$\cR(\lambda)\in \wt\cS^{(-1)}(\sM)$. \\

We are now ready to state the following further characterizations for the Stieltjes and inverse Stieljes families.

\begin{theorem}\label{zwanzig}
The following representations hold:
\begin{enumerate}
  \item  Let $\cQ\in\wt \cS(\sM)$.
Then there is a Hilbert space $\sH=\sM\oplus\sK$ and up to unitary equivalence a unique $\sM$-minimal nonnegative selfadjoint relation $\wt A$ in $\sH$ such that
\begin{equation}\label{einundzwan}
\cQ(\lambda)=-\cfrac{1}{\lambda}\left(P_\sM(\wt A-\lambda I)^{-1}\uphar\sM\right)^{-1}-I_\sM,\quad \lambda\in\dC\setminus\dR_+.
\end{equation}

  \item Let $\cR\in\wt \cS^{(-1)}(\sM)$. Then there is a Hilbert space $\sH=\sM\oplus\sK$ and up to unitary equivalence a unique $\sM$-minimal nonnegative selfadjoint relation $\wt B$ in $\sH$ such that
\begin{equation}\label{einundzwan2}
\cR(\lambda)=I_\sM -\left(P_\sM( I-\lambda\wt B)^{-1}\uphar\sM\right)^{-1},\quad \lambda\in\dC\setminus\dR_+.
\end{equation}
\end{enumerate}
Moreover, if $\cQ(\lambda)\in\wt\cS(\sM)$ is represented by means of $\wt A$ in \eqref{einundzwan}, then $-\cQ(\lambda)^{-1}$ admits the representation \eqref{einundzwan2} by means of $\wt B =\wt A^{-1}$.
\end{theorem}
\begin{proof}
(1) Let $\cQ\in\wt \cS(\sM)$. As indicated it follows from Lemma \ref{invthm} that $\cQ(\lambda)=\frac{\cR(\lambda)}{\lambda}$ for some
$\cR(\lambda)\in \wt\cS^{(-1)}(\sM)$. Hence, from Theorem \ref{obratno1} we obtain the following representation for $\cR(\lambda)$:
\[
 \cR(\lambda)= -\left(P_\sM\left(\wt A-\lambda I\right)^{-1}\uphar\sM\right)^{-1}-\lambda I_\sM,\quad \lambda\in\dC\setminus\dR_+.
\]
Dividing this expression by $\lambda$ yields the representation \eqref{einundzwan} for $\cQ(\lambda)$.

(2) From the resolvent formula \eqref{invres} we get $\lambda(\wt A-\lambda I)^{-1}=(I-\lambda \wt A^{-1})^{-1}-I$, $\lambda\in\dC\setminus\dR_+$.
Hence, \eqref{einundzwan} can be rewritten as
\[
 \cQ(\lambda) 
 =P_\sM(I-\lambda\wt A^{-1})^{-1}\uphar\sM\left(I_\sM-P_\sM(I-\lambda\wt A^{-1})^{-1}\uphar\sM\right)^{-1}
\]
and thus
\[
 -\cQ(\lambda)^{-1}=\left(P_\sM(I-\lambda\wt A^{-1})^{-1}\uphar\sM-I_\sM\right)\left( P_\sM(I-\lambda\wt A^{-1})^{-1}\uphar\sM \right)^{-1}.
\]
This leads to \eqref{einundzwan2} with the choices $\cR(\lambda)=-\cQ(\lambda)^{-1}$ and $\wt B=\wt A^{-1}$.
To complete the proof observe that $\wt A^{-1}$ is $\sM$-minimal if and only if $\wt A$ is $\sM$-minimal.
Indeed, if $T=-I+2(I+\wt A)^{-1}$ then $-T=-I+2(I+\wt A^{-1})^{-1}$; cf. \eqref{ATtrans}, \eqref{Atrans}.
Now the claim follows from Lemma \ref{TTlemma2}, since $-T$ is simple precisely when $T$ is simple.
\end{proof}

\subsection{Nevanlinna families as Weyl families of boundary relations}\label{dehamasn}
Let $\sK$ be a Hilbert space and define
\[
 J_\sK=\begin{bmatrix}0&-iI_\sK\cr iI_\sK&0\end{bmatrix}:\begin{array}{l}\sK\\\oplus\\\sK\end{array}\to\begin{array}{l}\sK\\\oplus\\\sK\end{array}.
\]
The operator $J_\sK$ is a fundamental symmetry ($J_\sK=J_\sK^*=J^{-1}_\sK$) in the Hilbert space $\sK^2:=\sK\oplus\sK$.
Define a linear transformation from $\sM^2\oplus\sK^2$ into $(\sM\oplus\sK)^2$ by
 \[
 \cJ:\left\{\begin{bmatrix}\f\cr \f'\end{bmatrix},\begin{bmatrix}f\cr f'\end{bmatrix}\right\}\mapsto \left\{\begin{bmatrix}\f\cr f\end{bmatrix},\begin{bmatrix}-\f'\cr f'\end{bmatrix}\right\},\quad \f,\f'\in\sM,\; f,f'\in \sK;
 \]
for this and for further results and details in this connection we refer to \cite{DHMS06}.
The formula $\wt A=\cJ(\Gamma)$ so-called \textit{main transform}
establishes a one-to-one correspondence between all unitary relations $\Gamma$ from the Kre\u{\i}n space $\left<\sK^2,J_\sK\right>$ into the Kre\u{\i}n space $\left<\sM^2,J_\sM\right>$ and all selfadjoint relations $\wt A$ in $(\sM\oplus\sK)^2$.
If $\Gamma$ is a unitary relation from the Kre\u{\i}n space $\left<\sK^2,J_\sK\right>$ into the Kre\u{\i}n space $\left<\sM^2,J_\sM\right>$, then
$\Gamma$ is called a \textit{boundary relation} of $S^*$, if $S=\ker\Gamma$. A boundary relation $\Gamma$  is called minimal if $\wt A=\cJ(\Gamma)$ is $\sM$-minimal.

Let $\cT=\dom\Gamma$. Then $\cT\subset\sK^2$ is a linear relation in $\sK$. Define $\sN_\lambda(\cT):=\ker(\cT-\lambda I)$, $\lambda\in\dC\setminus\dR$, and denote
\[
\wh \sN_\lambda(\cT)=\left\{\{g_\lambda,\lambda g_\lambda\}:\, g_\lambda\in\sN_\lambda(\cT)\right\}.
\]
The \textit{Weyl family} $\cW(\lambda)$ associated with $\Gamma$ is defined as follows \cite{DHMS06}:
\[
\cM(\lambda):=\Gamma\left(\wh \sN_\lambda(\cT)\right),\quad \lambda\in\dC\setminus\dR.
\]
According to \cite[Theorem 3.9]{DHMS06} there is (up to unitary equivalence) a unique minimal boundary relation $\Gamma:\sK^2\to\sM^2$
(where $\sK$ is some Hilbert space) whose the Weyl family
coincides with the given Nevanlinna family $\cM(\lambda)$ in $\sM$. In terms of compressed resolvent this means that there is a unique (up to unitary equivalence) a $\sM$-minimal selfdjoint relation $\wt A$ in the Hilbert space $\sM\oplus\sK$ (which is the \textit{the main transform} $\cJ(\Gamma)$) such that, see \cite{DHMS06},
\[
P_\sM(\wt A-\lambda I)^{-1}\uphar\sM=-(\cM(\lambda)+\lambda I)^{-1},\quad \IM \lambda\ne 0.
\]
The latter is equivalent to \eqref{opexpr}.

Let $\wt A$ be selfadjoint relation in the Hilbert space $\sH= \sM\oplus\sK$.
With $\lambda\in\dC$ define
\[
\cM(\lambda):=\left\{\left\{\f_\lambda,-\f'_\lambda\right\}:\left\{
\begin{bmatrix}\f_\lambda\cr f_\lambda\end{bmatrix},\begin{bmatrix}\f'_\lambda\cr \lambda f_\lambda\end{bmatrix}\right\}
\in\wt A\quad\mbox{for some }\; f_\lambda\in\sK, \; \f_\lambda,\f'_\lambda\in\sM \right\}.
\]
Hence,
\[
\cM(\lambda)=\left\{\left\{P_\sM(\wt A-\lambda I)^{-1}m,-\lambda P_\sM(\wt A-\lambda I)^{-1}m-m\right\}:\; m\in\sM\right\},\quad \lambda\in\rho(\wt A).
\]
The converse statement is also true \cite{DHMS06}.

Since the transformation
\[
\Gamma^T=\begin{bmatrix}0&iI\cr-iI &0 \end{bmatrix}\Gamma
\]
defines a unitary relation as well, its main transform $\wt A'=\cJ(\Gamma^T)$ is a selfadjoint relation.
The connection between $\wt A'$ and $\wt A$ is given by
\[
\left\{\,
 \left\{ \begin{bmatrix} h \\ f \end{bmatrix},
         \begin{bmatrix} h' \\ f' \end{bmatrix}
 \right\}\right\}\in\wt A\iff \left\{\left\{ \begin{bmatrix} -ih' \\f  \end{bmatrix},
         \begin{bmatrix} ih  \\ f' \end{bmatrix}
 \right\}\right\}\in \wt A'.
\]
The Weyl family of $\Gamma^T$ coincides with $-\cM^{-1}(\lambda)$; see \cite{DHMS06}. These facts and formulas lead to
various interpretations for the results on compressed resolvents appearing in the present paper.
We will discuss these connections in more detail elsewhere; cf. also \cite{ArlHassi_2015OT}.

The notion of boundary relation is a generalization of the notion of a space of boundary values or boundary triplet.
We will recalled some basic notions on them now as well, since they will be used in some later examples.
Let $S$ be a closed densely defined symmetric operator with equal defect numbers in $\sK$. Let $\sM$ be some Hilbert space,
$\Gamma_0$ and $\Gamma_1$ be linear mappings of $\dom(S^*)$ into
$\sM$. A triplet $\{\sM, \Gamma_0, \Gamma_1\}$ is called a space of
boundary values (s.b.v.) or an ordinary boundary triplet for $S^*$ \cite{DHMS06},
\cite{GG1} if

a) for all $x,y\in\dom (S^*)$ the Green's identity

 $$(S^*x, y)-(x, S^*y)=(\Gamma_1x, \Gamma_0y)_\sM-(\Gamma_0x,
\Gamma_1y)_\sM $$
holds, and

 b) the mapping
$$\dom S^*\ni x\mapsto \Gamma x= \{\Gamma_0x, \Gamma_1x\}\in \sM\times \sM$$ is surjective.

From this definition it follows that $\ker \Gamma_k\supset\cD(S),$
$k=0,1$, the operators
\[
A_0=S^*\uphar \ker \Gamma_0,\quad A_1=S^*\uphar \ker \Gamma_1
\]
are self-adjoint extensions of $S$, and moreover, they are
transversal
\[
\dom(S^*)= \dom(A_0 )+ \dom(A_1).
\]
The function
$$M(\lambda)(\Gamma_0x_\lambda)=\Gamma_1x_\lambda, \quad x_\lambda\in\mathfrak N_\lambda,$$
where $\mathfrak N_\lambda$ is a defect subspace of $S,$ is called
the Weyl function of the boundary triplet (cf. \cite{DHMS06}). If
\[
\gamma(\lambda):=\left(\Gamma_0\uphar\sN_\lambda\right)^{-1},
\]
then $M(\lambda)=\Gamma_1\gamma(\lambda)$.
The main transform of an ordinary boundary triplet is a selfadjoint operator in $\sM\oplus\sH$ and it is determined here by the formula
(cf. \cite{ArlHassi_2015OT})
\[
 \wt A\begin{bmatrix} \Gamma_0f\cr f  \end{bmatrix}=\begin{bmatrix}-\Gamma_1f\cr S^*f   \end{bmatrix}.
\]

\section{Further properties of the Stieltjes and inverse Stieltjes families}\label{sec4}

In this section a couple of special types of Stieltjes and inverse Stieltjes families are studied.
First an analog for the notion of an inner function is introduced in the setting of Stieltjes and inverse Stieltjes families
and then inner families in the classes $\wt\cS(\sM)$ and $\wt\cS^{(-1)}(\sM)$ are characterized.
Then we study Stieltjes and inverse Stieltjes families which admit certain scaling invariance properties.
We also investigate some qualitative properties of the following two mappings arising from Lemma \ref{invthm}: namely,
\[
 \Phi_+:\cQ(\lambda)\mapsto -\cfrac{\cQ(\lambda)^{-1}}{\lambda} \quad \text{and} \quad
 \Phi_-:\cR(\lambda)\mapsto -\lambda \cR(\lambda)^{-1}.
 \]
Notice that by Lemma \ref{invthm} $\Phi_+$ maps a function $\cQ\in\wt\cS(\sM)$ back to $\wt\cS(\sM)$,
while $\Phi_-$ maps a function $\cR\in\wt\cS^{(-1)}(\sM)$ back to $\wt\cS^{(-1)}(\sM)$.
The fixed points of the mapping $\Phi_+$ in the class $\wt\cS(\sM)$ and the mapping $\Phi_-$ in the class $\wt\cS^{(-1)}(\sM)$
will be described.

\subsection{Inner functions in the Stieltjes and inverse Stieltjes classes}\label{innerfamily}

Recall that an operator valued Schur function is said to be inner/co-inner/bi-inner if almost everywhere on the unit disk the non-tangential limit values
to the unit circle $\dT$ are, respectively, isometric/co-isometric/unitary. It is proved in \cite{ArlHassi_2018} that the function $\Omega$ of
the class $\cRS(\sM)$ is inner if and only if it admits the representation
\begin{equation}\label{dblinne}
\Omega(z)=(zI +\wt D)(I+z\wt D)^{-1},\quad z\in\dC\setminus\{(-\infty,-1]\cup[1,+\infty)\},
\end{equation}
where $\wt D$ is a selfadjoint contraction in $\sM$.

The Stieltjes class $\cS(\sM)$ and the inverse Stieltjes class $\wt\cS^{(-1)}(\sM)$ are connected to the class $\cRS(\sM)$
as described in Lemma \ref{TH1}. Notice that (cf. \eqref{NevTrans2})
\[
 \RE \lambda=\cfrac{|z|^2-1}{|z+1|^2}, \quad \lambda=\cfrac{z-1}{z+1}.
\]
In particular, the transform $z\to \cfrac{z-1}{z+1}$ maps the nonreal part of the unit circle $\dT\setminus\{1,-1\}$ bijectively
onto the set $\{iy:\, y\in \dR, \; y\neq 0\}$, i.e. the imaginary axis excluding the origin.
This motivates the following definition.

\begin{definition}\label{incrs}
A family $S$ from the class $\wt\cS(\sM)\;(\wt\cS^{(-1)}(\sM))$ is said to be inner if
$$\RE(S(iy)f,f)=0,\quad f\in\dom S(iy),$$
holds for all $y\in\dR\setminus\{0\}$.
\end{definition}

\begin{theorem}\label{vildyen}
All inner families in the Stieljes and inverse Stieltjes classes are described as follows:
\begin{enumerate}
\item The inner families from the class $\wt\cS(\sM)$ are of the form
\begin{equation} \label{vildyen1}
 \cQ(\lambda)=-\lambda^{-1}\,\cB,\quad \lambda\in\dC\setminus\dR_+,
\end{equation}
where $\cB$ runs through the set of all nonnegative selfadjoint relations in $\sM$.
\item The inner families from the class $\wt\cS^{(-1)}(\sM)$ are of the form
\begin{equation}\label{vildyen2}
 \cR(\lambda)={\lambda}\,\cC,\quad  \lambda\in\dC\setminus\dR_+,
\end{equation}
where $\cC$ runs through the set of all nonnegative selfadjoint relations in $\sM$.
\end{enumerate}
\end{theorem}
\begin{proof}
Let $\cQ$ be from the class $\wt\cS(\sM)$. Then by Lemma \ref{TH1} the function
\[
\Omega(z):=I-2\left(I+\cQ\left(\cfrac{z-1}{1+z}\right)\right)^{-1},\quad
 z\in\dC\setminus\left((-\infty,-1]\cup[1,+\infty)\right),
\]
belongs to the class $\cRS(\sM)$ and $\cQ(\lambda)=(I+\Omega(z))(I-\Omega(z))^{-1}$ with $z=\frac{1+\lambda}{1-\lambda}$.
Since
\begin{equation}\label{In}
 \RE(\cQ(\lambda)f,f)=\left([I-\Omega(z)^*\Omega(z)](I-\Omega(z))^{-1}f,(I-\Omega(z))^{-1}f\right), \quad f\in\dom \cQ(\lambda),
\end{equation}
we conclude that $\cQ$ is inner if and only if $\Omega$ is inner. Therefore, $\Omega$ has the expression
\eqref{dblinne} and hence its transform $\cQ$ takes the form \eqref{vildyen1} with $\cB=(I+\wt D)(I-\wt D)^{-1}$.

Similarly using the transform \eqref{formula2} from Lemma \ref{TH1}, one derives from \eqref{dblinne} the formula
\eqref{vildyen2} with $\cC=(I-\wt D)(I+\wt D)^{-1}$.
\end{proof}

\begin{remark}
\label{constanto}
Let $P$ be an orthogonal projection in $\sM$, then the constant family
\[
\cM(\lambda)=\{\{Pf, (I-P)f\}:\; f\in \sM\}
\]
is an inner Stieltjes family and an inner inverse Stieltjes family simultaneously.
\end{remark}

\subsection{Scale invariant Stieltjes and inverse Stieltjes families}\label{scaleinv11}

\begin{definition}\label{scainv}
A Nevanlinna family $\cM$ in the Hilbert space $\sM$ is said to be scale invariant if for some $c\in\dR_+$ and for some $p\in\{0,1,-1\}$  the relation
\[
\cM(c\lambda)=c^p\cM(\lambda)
\]
holds for all $\lambda\in\cmr$.
\end{definition}

As shown in the next theorem scale invariant Stieltjes and inverse Stieltjes families admit their own specific characterizations
for each of the choices $p=0$, $p=1$, and $p=-1$.

\begin{theorem}\label{scinv}
Let $\sM$ be a Hilbert space and let $c\in\dR_+$ and assume that $c\neq 1$.
\begin{enumerate}
\item Each Stieltjes family in $\sM$, satisfying the equality $\cQ(c\lambda)=\cQ(\lambda)$ $\forall \lambda\in\dC\setminus\dR_+ $ is of the form
$\cQ(\lambda)\equiv\cA$, where $\cA$ is a nonnegative selfadjoint relation.
\item Each inverse Stieltjes family in $\sM$, satisfying the equality $\cR(c\lambda)=\cR(\lambda)$ $\forall \lambda\in\dC\setminus\dR_+, $ is of the form
$\cR(\lambda)\equiv-\cA$, where $\cA$ is a nonnegative selfadjoint relation.
\item Each Stieltjes family in $\sM$, satisfying the equality $\cQ(c\lambda)=c^{-1}\cQ(\lambda)$ $\forall \lambda\in\dC\setminus\dR_+, $ is of the form
$$\cQ(\lambda)=-\lambda^{-1} \cB,$$ where $\cB$ is a nonnegative selfadjoint relation.
\item Each inverse Stieltjes family in $\sM$, satisfying the equality $\cR(c\lambda)=c\cR(\lambda)$ $\forall \lambda\in\dC\setminus\dR_+, $ is of the form
$$\cR(\lambda)=\lambda\,\cC,$$
  where $\cC$ is a nonnegative selfadjoint relation.
\item Each Stieltjes family $\cQ$ in $\sM$, satisfying the equality $\cQ(c\lambda)=c\cQ(\lambda)$ $\forall \lambda\in\dC\setminus\dR_+, $ and each inverse Stieltjes family $\cR$ in $\sM$, satisfying the equality $\cR(c\lambda)=c^{-1}\cR(\lambda)$ $\forall \lambda\in\dC\setminus\dR_+, $  is of the form
$$\cQ(\lambda)\equiv\left\{\left\{Pf,(I-P)f\right\}\right\},$$
where $P$ is an orthogonal projection in $\sM$.
\end{enumerate}
\end{theorem}

\begin{proof}
First notice that if in $z=\frac{1+\lambda}{1-\lambda}$ one replaces
$\lambda$ by $c\lambda$, then
\begin{equation}\label{aexpc}
  \frac{1+c\lambda}{1-c\lambda} = \cfrac{z+a}{1+z a}, \quad \text{where } a=\cfrac{1-c}{1+c} \in (-1,1), \quad c>0.
\end{equation}
Hence, if $\cQ(\lambda)\in \wt\cS(\sM)$ or $\cR(\lambda)\in\wt\cS^{(-1)}(\sM)$ is represented by \eqref{formula1} or \eqref{formula2} with $\Omega(z)\in\cRS(\sM)$,
then the families $\cQ(c\lambda)$ and $\cR(c\lambda)$ are represented by means of $\Omega\left(\cfrac{z+a}{1+z a}\right)$.
Notice also that the condition $c\neq 1$ is equivalent to $a\neq 0$.

On the other hand, if $\cQ\in\wt\cS(\sM)$ and  $\Omega(z)$ (see Lemma \ref{TH1}) is defined by
\begin{equation}\label{Omega2}
\Omega(z)=I-2\left(I+\cQ\left(\cfrac{z-1}{z+1}\right)\right)^{-1}\in\cRS(\sM),
\end{equation}
then the function
\[
\Psi(z)=I-2\left(I+c\,\cQ\left(\cfrac{z-1}{z+1}\right)\right)^{-1} ,\quad  z\in\dC\setminus\{(-\infty,-1]\cup[1,+\infty)\},
\]
can be expressed with $\Omega(z)$ in the form
\[
 \Psi(z)=\left(\Omega(z)-a I_\sM\right)\left(I_\sM-a\Omega(z)\right)^{-1}, 
\]
where $a$ is given by \eqref{aexpc}. Similarly, there is the identity
\[
  I-2\left(I+\frac{1}{c}\,\cQ\left(\cfrac{z-1}{z+1}\right)\right)^{-1}
  =\left(\Omega(z)+a I_\sM\right)\left(I_\sM+a\Omega(z)\right)^{-1}.
\]
We now prove the assertions (1) -- (5) in three parts.

\textit{Verification of (1) \& (2).}
Using the transformations in Lemma \ref{TH1} and the observations just made above, the equality $\cQ(c\lambda)=\cQ(\lambda)$
with $\cQ\in\wt\cS(\sM)$ and the equality $\cR(c\lambda)=\cR(\lambda)$ with $\cR\in\wt\cS^{(-1)}(\sM)$ for all
$\lambda\in\dC\setminus\dR_+$ with $\cR\in\wt\cS^{(-1)}(\sM)$ is equivalent to the equality
\[
 \Omega\left(\cfrac{z+a}{1+z a}\right)=\Omega(z)\quad \forall z\in\dC\setminus\{(-\infty,-1]\cup[1,+\infty)\}
\]
with $\Omega\in\cRS(\sM)$.
According to \cite[Proposition 6.13]{ArlHassi_2018} this equality with $a\neq 0$ has only constant solutions.
Hence, $\Omega(z)\equiv \Omega(0)$ is a selfadjoint contraction in $\sM$. Here
\[
 -I\le \Omega(0)\le I \quad\Longleftrightarrow\quad 2(I-\Omega(0))^{-1}\geq I,
\]
and now applying Lemma \ref{TH1} once again the statements (1) and (2) follow.

\textit{Verification of (3) \& (4).}
By the above observations, the equality $\cQ(c\lambda)=c^{-1}\cQ(\lambda)$ $\forall \lambda\in\dC\setminus\dR_+$
with $\cQ\in\wt\cS(\sM)$ is equivalent to the equality
\begin{equation}\label{IdenOmega2}
  \Omega\left(\cfrac{z+a}{1+z a}\right)= \left(\Omega(z)+a I_\sM\right)\left(I_\sM+a\Omega(z)\right)^{-1} ,\quad z\in\dC\setminus\{(-\infty,-1]\cup[1,+\infty)\},
\end{equation}
where $\Omega(z)$ is given by \eqref{Omega2}.
It is proved in \cite[Theorem 6.18]{ArlHassi_2018} that the solutions to the equation \eqref{IdenOmega2} with $a\neq 0$
consist of the inner functions from the class $\cRS(\sM)$.
Since $\cQ$ is inner if and only if $\Omega$ is inner (see \eqref{In}), we conclude from Theorem \ref{vildyen})
that $\cQ(\lambda)=-\lambda^{-1} \cB$, where $\cB$ is a nonnegative selfadjoint realtion in $\sM$.
On the other hand, the equality $\cR(c\lambda)=c\cR(\lambda)$ $\forall\lambda\in\dC\setminus\dR_+$ with $\cR\in\wt\cS^{(-1)}(\sM)$
is equivalent to the equality
\[
-\cR^{-1}(c\lambda)={c}^{-1}\left(-\cR^{-1}(\lambda)\right)\;\forall\lambda\in\dC\setminus\dR_+,
\]
where $-\cR^{-1}\in\wt\cS(\sM)$. This implies that $\cR(\lambda)=\lambda\cC$, where $\cC=\cB^{-1}$ is a nonnegative selfadjoint relation.

\textit{Verification of (5).}
Suppose that the Stieltjes family $\cQ$ satisfies the equality $\cQ(c\lambda)=c\cQ(\lambda)$ $\forall \lambda\in\dC\setminus\dR_+$
with some $c>0$, $c\neq 1$. Recall that $\cQ(\lambda)$ is a Nevanlinna family with a holomorphic continuation to the negative semi-axis.
Then $\cQ(x)$ is a monotonically nondecreasing function of the negative semi-axis and by assumption $\cQ(x)\geq 0$ for all $x<0$.
Now fix $x<0$. By assumption $cQ(x)=Q(cx)$ with $c\neq 1$.
This condition implies in particular that $\dom \cQ(x)=\dom \cQ(cx)$ and, moreover, that $\cQ(x)$ and $\cQ(cx)$ generate closed nonnegative forms with the same form domain.
If, for instance, $c>1$ then $cx<x$ and using monotonicity we conclude that
\[
 0\leq c\cQ(x)=\cQ(cx)\leq \cQ(x),
\]
and hence $0\leq (c-1)\cQ(x) \leq 0$; for proper meaning of monotonicity in this general setting we refer to  \cite{BHSW2010}.
This implies that for all $f\in \dom \cQ(x)$ one has $\cQ(x)f=0$.
Thus $\dom \cQ(x)=\ker \cQ(x)$, i.e., $\cQ(x)$ is a singular relation.
Since $\cQ(x)$ is selfadjoint, $\mul \cQ(x)=\dom \cQ(x)^\perp$
and we conclude that
\[
 \cQ(x)=\left\{\{Pf, (I_\sM-P)f\}:\; f\in\sM\right\},
\]
where $P$ stands for the orthogonal projection onto $\cdom \cQ(x)=\dom\cQ(x)$.
However, $\cQ(x)$ is continuous (by assumption even holomorphic) as a function of $x$ and therefore the projector $P$ cannot depend on $x<0$
(by a general principle concerning continuous paths of projectors on connected sets). Furthermore, by holomorphy we conclude that
$\cQ(\lambda)\equiv\left\{\left\{Pf,(I-P)f\right\}\right\}$.

In the same way one treats the case, where $0<c<1$. Therefore, the statement for Stieltjes families is proven.
The statement concerning the inverse Stieltjes families is obtained by passing to the inverses.
\end{proof}

\begin{remark}
(i) There is a connection between assertions (1) and (4) and assertions (2) and (3) in Theorem \ref{scinv}, which can be seen by means of Lemma \ref{invthm}.
For instance, assume that $\cQ\in\wt\cS(\sM)$ satisfies $\cQ(c\lambda)=\cQ(\lambda)$ $\forall \lambda\in\dC\setminus\dR_+$.
Then by Lemma \ref{invthm} the function $\cR(\lambda):=\lambda \cQ(\lambda)$ belongs to $\wt\cS^{(-1)}(\sM)$ and it satisfies
\[
 \cR(c\lambda)=c\lambda \cQ(c\lambda)= c\lambda \cQ(\lambda)= c\cR(\lambda).
\]
Conversely, if $\cR(c\lambda)=c\cR(\lambda)$ then the function $\cQ(\lambda):=\lambda^{-1}\cR(\lambda)\in\wt\cS(\sM)$ satisfies the identity
$\cQ(c\lambda)=\cQ(\lambda)$. Hence, if e.g. (1) holds, then $\cQ(\lambda)\equiv\cA$, where $\cA$ is a nonnegative selfadjoint relation.
This means that $\cR(\lambda)=\lambda \cQ(\lambda)=\lambda \cA$. Similarly, we get the connection between assertions (2) and (3) in Theorem \ref{scinv}.

(ii) The intersection of the Stieltjes and inverse Stieltjes classes in $\sM$ coincides with the class
of constant functions of the form appearing in part (5) of Theorem \ref{scinv}:
\[
 \cQ(\lambda)=\left\{\{Pf, (I_\sM-P)f\}:\; f\in\sM\right\}.
\]
This is clear from the conditions $\cQ(x)\geq 0$ for Stieltjes families and  $\cQ(x)\leq 0$ for the inverse Stieltjes families, which imply
that $(\f',\f)=0$ for all $\{\f,\f'\}\in\cQ(x)$. Hence, $\dom \cQ(x)=\ker \cQ(x)$ and the claim follows; cf. the proof of (5) given above.

(iii) The statement (5) in Theorem \ref{scinv} can be also obtained by making a connection to the class $\cRS(\sM)$.
Namely, the equality $\cQ(c\lambda)=c\cQ(\lambda)$ for the Stieltjes family $\cQ$ in $\sM$ with $\lambda\in\dC\setminus\dR_+$ and $c\neq 1$
is equivalent to the equality
\[
\left(\Omega(z)-a I_\sM\right)\left(I_\sM-a\Omega(z)\right)^{-1}=\Omega\left(\cfrac{z+a}{1+z a}\right)\;\forall z\in\dC\setminus\{(-\infty,-1]\cup[1,+\infty)\},
\]
where $\Omega(z)$ is given by \eqref{Omega2} and $a\neq 0$. It is proved in \cite[Theorem 6.19]{ArlHassi_2018} that
the only solutions to this equation are the constant functions $\Omega(z)\equiv D$, $z\in \dC\setminus\{(-\infty,-1]\cup[1,+\infty)\}$, where
$D$ is a fundamental symmetry in $\sM$. Because $D=I-2P$, where $P$ is an orthogonal projection in $\sM$, this yields
\[
 \cQ(\lambda)=\left\{\{Pf, (I_\sM-P)f\}:\; f\in\sM\right\}.
\]
\end{remark}

\subsection{The mappings $\Phi_+$ and $\Phi_-$, and realizations of their fixed points}\label{sec4.3}

Recall that by Lemma \ref{invthm} the transformation
\[
\wt\cS(\sM)\ni \cQ(\lambda)\stackrel{{{\mathbf \Phi_+}}}\mapsto \wt \cQ(\lambda):=-\cfrac{\cQ(\lambda)^{-1}}{\lambda}\in\wt\cS(\sM)
\]
is well defined mapping in the Stieltjes class.
In fact, $\mathbf{\Phi_+}$ is an automorphism of the class $\wt\cS(\sM)$.
Analogously, the transformation
\[
\wt\cS^{-1}(\sM)\ni \cR(\lambda)\stackrel{{{\mathbf \Phi_-}}}\mapsto \wt \cR(\lambda):=-\lambda \cR(\lambda)^{-1}\in\wt\cS^{(-1)}(\sM)
\]
is an automorphism of the class $\wt\cS^{-1}(\sM)$.
Here the main purpose is to find the fixed points of these two mappings.

\begin{proposition}\label{PropFixedpoint}
Let the mappings $\Phi_+:\wt\cS(\sM)\to \wt\cS(\sM)$ and $\Phi_-:\wt\cS^{(-1)}(\sM)\to \wt\cS^{(-1)}(\sM)$ be as defined above.
Then with $\lambda\in \dC\setminus \dR_+$:
\begin{enumerate}
\item the mapping $\Phi_+$ has a unique fixed point
\[
\cQ_0(\lambda)=\cfrac{i}{\sqrt{\lambda}}\,I_\sM,\quad \cQ_0(-1)=I_\sM;
\]
\item
the mapping $\Phi_-$ has a unique fixed point
\[
\cR_0(\lambda)=i{\sqrt{\lambda}}\,I_\sM,\quad \cR_0(-1)=-I_\sM.
\]
\end{enumerate}
  \end{proposition}
\begin{proof}
(1) Let $\cQ\in\wt\cS(\sM)$ and consider the equation
\begin{equation}\label{solQ}
 \cQ(\lambda)=-\cfrac{\cQ(\lambda)^{-1}}{\lambda}, \quad  \lambda\in \dC\setminus \dR_+.
\end{equation}
Then $\ker \cQ(\lambda)=\mul \cQ(\lambda)$ and $\dom \cQ(\lambda)=\ran \cQ(\lambda)$.
In particular, since $\cQ(x)$ is selfadjoint for $x<0$ one has $\ker \cQ(\lambda)\perp \mul \cQ(\lambda)$ and hence
$\ker \cQ(\lambda)=\mul \cQ(\lambda)=\{0\}$. Moreover, and $\dom \cQ(x)=\ran \cQ(x)$ implies that $\cQ(x)$ is a bounded selafadjoint
operator. Then by holomorphy $\cQ(\lambda)$ is bouhded when $\IM \lambda$ is sufficiently small and thus by \cite[Proposition~4.18]{DHMS06}
$\cQ(\lambda)\in\mathbf{B}(\sM)$ for all $\dC\setminus \dR_+$. Moreover, $\ker \cQ(\lambda)=\mul \cQ(\lambda)=\{0\}$ for all $\dC\setminus \dR_+$.
This implies that \eqref{solQ} is equivalent to
\begin{equation}\label{solQ2}
 (\cQ(\lambda))^2=-\cfrac{1}{\lambda}.
\end{equation}
Since $\cQ(\lambda)$ is a Nevanlinna function and holomorphic on the simply connected set $\cD\setminus\dR_+$, the unique solution
to the equation \eqref{solQ2} is the function $\cQ(\lambda)=\cfrac{i}{\sqrt{\lambda}}\,I_\sM$, $\lambda\in\dC\setminus\dR_+$.

(2) This is obtained from (1) by passing to the inverses.
\end{proof}

\begin{remark}
The equation \eqref{solQ} is equivalent to the equation
\begin{equation}\label{solQ3}
 \lambda\cQ(\lambda)=-\cQ(\lambda)^{-1}, \quad  \lambda\in \dC\setminus \dR_+.
\end{equation}
The functions on both sides of \eqref{solQ3} are inverse Stieltjes families by Lemma \ref{invthm}.
Applying the connection \eqref{formula2} in Lemma \ref{TH1} to the equation \eqref{solQ3} leads to equivalent condition
\[
 \left(zI_\sM -\Omega(z)\right)\left(I_\sM -z\Omega(z)\right)^{-1}=\Omega(z).
\]
It is shown in \cite[Proposition~6.6]{ArlHassi_2018} that the unique solution to this last equation in the class $\cRS(\sM)$ is the function
\[
 \Omega_0(z)=\frac{z I_\sM}{1+\sqrt{1-z^2}}, \quad z\in\dC\setminus\{(-\infty,-1]\cup[1,+\infty)\}.
\]
A straightforward calculation shows that $\Omega_0(z)$ and $\cQ_0(\lambda)$ in Proposition \ref{PropFixedpoint}
are connected by the formula \eqref{formula1} in Lemma \ref{TH1}.
\end{remark}

To complete this section we construct realizations for the functions $\cQ_0$ and $\cR_0$ in Proposition \ref{PropFixedpoint}.
In this way we simultaneously demonstrate the general results obtained in this paper.

\bigskip

(A) We demonstrate Theorems~\ref{TH0},~\ref{obratno1}, and \ref{stieltacc} by treating an $L_2$-model for the functions
$\cQ_0$ and $\cR_0$. Let $\sM$ be a Hilbert space and consider the weighted Hilbert space
$\sH_0=L_2(\sM,\dR_+,\rho_0(t))$, where
$$\rho_0(t)=\cfrac{2}{\pi}\,\cfrac{1}{1+t^2},\quad t\in\dR_+.$$
The inner product is given by
\[
(f,g)_{\sH_0}=\int\limits_{\dR_+}(f(t),g(t))_\sM\,\rho_0(t)dt.
\]
The Hilbert space $\sM$ we can identify with the subspace of $\sH_0$ consisting of constant functions.
It easy to see that
\[
P_\sM f(t)=\int\limits_{\dR_+}f(t)\,\rho_0(t)dt.
\]
Let $\wt A_0$ be the operator of multiplication by the squared independent variable:
\[
\wt A_0f(t)=t^2f(t),\quad \dom \wt A_0=\left\{f\in \sH_0:\int\limits_{\dR_+}t^4||f(t)||^2_\sM\,\rho_0(t)dt<\infty \right\}.
\]
Then
the operator $\wt A_0$ is selfadjoint, nonnegative, and
$$(\wt A_0-\lambda)^{-1}g(t)= \cfrac{g(t)}{t^2-\lambda},\quad g\in \sH_0,\quad \lambda\in\dC\setminus\dR_+.$$

Let $g_0(t)=g_0,$ $t\in\dR_+$, $g_0\in\sM$. Then for $\lambda\in\dC\setminus\dR_+$, using residues, one obtains
\[
P_\sM(\wt A_0-\lambda I)^{-1}g_0(t)=\int\limits_{\dR_+}\cfrac{g_0}{t^2-\lambda}\,\rho_0(t)dt=\cfrac{2g_0}{\pi}\int\limits_{\dR_+}\cfrac{dt}{(t^2-\lambda)(1+t^2)}
=-\cfrac{g_0}{\lambda+i\sqrt{\lambda}},
\]
where the branch $\IM\lambda>0\Longrightarrow\IM\sqrt{\lambda}>0$ is chosen.
Therefore,
\[
P_\sM(\wt A_0-\lambda)^{-1}\uphar\sM=-\left(\cR_0(\lambda)+\lambda I_\sM\right)^{-1},
\]
where $\cR_0(\lambda)=i\sqrt{\lambda}I_\sM$.

Clearly, $\ker  \wt A_0=\{0\}$. Then according to Remark \ref{relvsop} the operators $\wh A_0=\sJ_\sM(\wt A_0)$ and $\wh B_0=\sP_\sM(\wt A_0)$ are of the form
\[
\wh A_0\left(f(t)-\cfrac{2}{\pi}\int\limits_{\dR_+}\cfrac{it^2+1}{t^2+1}\, f(t)dt\right)
 =t^2 f(t)+\cfrac{2}{\pi}\int\limits_{\dR_+}\cfrac{i-t^2}{t^2+1}\, f(t)dt,\quad f(t)\in\dom \wt A_0,
\]
\[
\wh B_0\left(f(t)+\cfrac{2}{\pi}\int\limits_{\dR_+}\cfrac{t^2-1}{t^2+1}\, f(t)dt\right)
 =t^2 f(t)-\cfrac{2}{\pi}\int\limits_{\dR_+}\cfrac{t^2-1}{t^2+1}\, f(t)dt,\quad f(t)\in\dom \wt A_0.
\]
Now by Theorem \ref{TH0} we have
\[
P_\sM(\wh A_0-\lambda)^{-1}\uphar\sM=-\left(\cfrac{i}{\sqrt\lambda}+\lambda \right)^{-1}\, I_\sM,\quad \lambda\in\cmr,
\]
\[
P_\sM(\wh B_0-\lambda)^{-1}\uphar\sM=\left(\cfrac{i}{\sqrt\lambda}-\lambda \right)^{-1}\, I_\sM,\quad \RE\lambda<0.
\]

(B) Now we treat a realization by means of a second order differential operator on the semi-axis $\dR_+$.
Let $\sH_0=\sM\oplus L_2(\sM,\dR_+,dt)$ and let $H^2(\sM,\dR_+)$ be the Sobolev space.
Define
\begin{equation}\label{xfcnyckex}
\wt \cA_0\begin{bmatrix}u(0)\cr u(x)\end{bmatrix}=\begin{bmatrix}-u'(0)\cr -u''(x)\end{bmatrix},\; u(x)\in H^2(\sM,\dR_+).
\end{equation}
Consider a closed symmetric nonnegative operator $S_0$ in $L_2(\sM,\dR_+,dt)$:
\[
S_0u=-u'',\quad \dom S_0=\left\{u\in H^2(\sM,\dR_+):\, u(0)=u'(0)=0\right\}.
\]
The adjoint operator $S^*_0$ is given by
\[
S^*_0u=-u'',\quad \dom S^*_0=H^2(\sM,\dR_+).
\]
Define a pair of boundary mappings by
\[
\Gamma_0 u=u(0),\quad \Gamma_1(0)=u'(0),\quad u\in H^2(\sM,\dR_+).
\]
Then $\{\sM,\Gamma_0,\Gamma_1\}$ is a boundary triplet for $S_0^*$, see Subsection \ref{dehamasn}, and
\[
(S^*u,u)-(\Gamma_1u,\Gamma_0 u)_\sM=\int\limits_{\dR_+}||u'(x)||^2_\sM dx\ge 0\quad\forall u\in H^2(\sM,\dR_+).
\]
The operator $\wt A_0$ \eqref{xfcnyckex} is the main transform of the boundary triplet $\{\sM,\Gamma_0,\Gamma_1\}$. It is selfadjoint and nonnegative in $\sH_0$.

Let $\lambda\in\dC\setminus\dR_+$. The system of equations
\[
\left\{\begin{array}{l} -u'(0)-\lambda u(0)=h\\
-u''(x)-\lambda u(x)=0\end{array}\right.
\]
has a unique solution $\begin{bmatrix}\cfrac{-h}{i\sqrt{\lambda}+\lambda}\cr\cfrac{-e^{i\sqrt{\lambda}}h}{i\sqrt{\lambda}+\lambda} \end{bmatrix}$, where $h\in\sM$.
Hence,
\[
P_\sM\left(\wt \cA_0-\lambda I\right)^{-1}h=-\cfrac{h}{i\sqrt{\lambda}+\lambda},
\]
see Theorem \ref{TH0}. Now the transforms $\wh \cA_0=\sJ_\sM(\wt \cA_0)$ and $\wh \cB_0=\sP_\sM(\wt \cA_0)$ are of the form
\[
\wh\cA_0\begin{bmatrix}iu'(0)\cr u(x)\end{bmatrix}=\begin{bmatrix}iu(0)\cr- u''(x)\end{bmatrix},\qquad
\wh\cB_0\begin{bmatrix}-u'(0)\cr u(x)\end{bmatrix}=\begin{bmatrix}u(0)\cr- u''(x)\end{bmatrix},\quad u(x)\in H^2(\sM,\dR_+).
\]
It remains to note that (see Theorems~\ref{TH0},~\ref{obratno1}, and \ref{stieltacc})
\[
P_\sM(\wh\cA_0-\lambda I)^{-1}\uphar\sM=-\cfrac{1}{\cfrac{i}{\sqrt{\lambda}}+\lambda}\, I_\sM,\; \lambda\in\cmr
\]
\[
P_\sM(\wh\cB_0-\lambda I)^{-1}\uphar\sM=\cfrac{1}{\cfrac{i}{\sqrt{\lambda}}-\lambda}\, I_\sM,\;\RE\lambda<0.
\]

\begin{appendix}

%

\section{Discrete-time systems and their transfer functions} \label{AppendA}

Let $\sM,\sN$, and $\sH$ be separable Hilbert spaces. A linear
system
$$\tau=\left\{\begin{bmatrix} D&C \cr
B&F\end{bmatrix};\sM,\sN, \sK\right\}$$
with bounded linear operators
$F$, $B$, $C$, $D$  of the form
\[
 \left\{
 \begin{array}{l}
    \sigma_k=Ch_k+D\xi_k,\\
    h_{k+1}=Fh_k+B\xi_k
\end{array}
\right. \qquad k\in\dN_0,
\]
where $\{\xi_k\}\subset \sM$, $\{\sigma_k\}\subset \sN$,
$\{h_k\}\subset \sK$ is called a \textit{discrete time-invariant
system}, cf.\cite{A}. The Hilbert spaces $\sM$ and $\sN$ are called the input and
the output spaces, respectively, and the Hilbert space $\sK$ is
called the state space.
The transfer function of the system $\tau$ is defined by
\begin{equation}
\label{TrFu}
\Omega(z):=D+z C(I_{\sK}-z F)^{-1}B
\end{equation}
and it is holomorphic in a neighborhood of the origin.

Associate with $\tau$ the block operator matrix
\[
T=\begin{bmatrix} D&C \cr B&F\end{bmatrix} :
\begin{array}{l} \sM \\\oplus\\ \sK \end{array} \to
\begin{array}{l} \sN \\\oplus\\ \sK \end{array}.
\]
If $T$ is contractive, then the corresponding discrete-time system
is said to be \textit{passive} \cite{A}. If $T$ is unitary, then the
system is called conservative.
The transfer function
of a passive system $\tau$ belongs to the \textit{Schur class}
$\cS(\sM,\sN)$ of all holomorphic and contractive
$\bB(\sM,\sN)$-valued functions on the unit disk $\dD$.

The subspaces
\[
\sK^{c}:=\cspan\left\{F^n B\sM:\;n\in\dN\cup\{0\}\right\}, \quad \sK^{o}:=\cspan\left\{F^{*n} C^*\sN:\;n\in\dN\cup\{0\}\right\}
\]
are called the controllable and observable subspaces of $\tau$, respectively.
If $\sK^{c}=\sK$ (respectively, $\sK^{o}=\sK$), then the system $\tau$ is called controllable (resp. observable). If
\[
{\rm{clos}}\left\{\sK^{c}+\sK^{o}\right\}=\sK,
\]
then $\tau $ is called simple, and if $\sK^{c}=\sK^{o}=\sK$, i.e., $\tau$ is controllable and observable system, then the system is called minimal.

The resolvent $\cR_T(z)=(I-zT)^{-1}$ of the block operator $T$ is of the form (the Schur-Frobenius formula for the resolvent):
\begin{equation}
\label{Sh-Fr1}
\begin{array}{l}
\cR_T(z)=
\begin{bmatrix}(I-z\Omega(z))^{-1}&z(I-z\Omega(z))^{-1}C\cR_F(z)\cr
z\cR_F(z)B(I-z\Omega(z))^{-1}&\cR_F(z)\left(I+zB(I-z\Omega(z))^{-1}C\cR_F(z)\right)
\end{bmatrix},\\[4mm]
\quad z^{-1}\in\rho(T)\cap\rho(F),\; z\neq 0,
\end{array}
\end{equation}
where $\Omega$ is given by \eqref{TrFu}.
In particular, the Schur-Frobenius formula \eqref{Sh-Fr1} shows that
\begin{equation}
\label{sch-fr}
P_\sM(I-zT)^{-1}\uphar\sM=(I_\sM-z\Omega(z))^{-1},\quad z\in\dD.
\end{equation}
Besides, the resolvent formula \eqref{Sh-Fr1} yields
\begin{equation}\label{SPAN11}
\begin{array}{l}
\cspan\left\{T^n\sM: n\in\dN_0\right\}=\cspan\left\{(I-z T)^{-1}\sM:z\in \cU\right\}\\
=\sM\oplus\cspan\left\{F^n B\sM:\; n\in\dN\cup\{0\}\right\},
\end{array}
\end{equation}
\begin{equation}\label{SPAN22}
\begin{array}{l}
\cspan\left\{T^{*n}\sM: n\in\dN_0\right\}=\cspan\left\{(I-z T^*)^{-1}\sM:z\in \cU\right\}\\
=\sM\oplus\cspan\left\{F^{*n} C^*\sM:\; n\in\dN\cup\{0\}\right\}
\end{array}
\end{equation}
for any small neighborhood $\cU$ of the origin.

For a passive selfadjoint system $\tau=\left\{ \begin{bmatrix} D&C\cr C^*&F\end{bmatrix};\sM,\sM,\sK\right\}$
the controllable and observable subspaces coincide.

\begin{definition}\label{rsmA} \cite{ArlHassi_2018}.
Let $\sM$ be a Hilbert space. A $\bB(\sM)$-valued Nevanlinna function $\Omega$ holomorphic on $\dC\setminus\{(-\infty,-1]\cup[1,+\infty)\}$ is said to belong to the class $\cRS(\sM)$ if $-I\le
\Omega(x)\le I$ for $x\in (-1,1)$.
\end{definition}

Let $\Omega\in\cRS(\sM)$. By \cite[Theorem 5.1, Proposition 5.6]{AHS_CAOT2007}
there exists up to the unitary equivalence a unique minimal passive
selfadjoint system
$$\tau=\left\{ \begin{bmatrix} D&C\cr C^*&F\end{bmatrix};\sM,\sM,\sK\right\},$$
whose transfer function coincides with $\Omega(z)$, i.e.,
\[
\Omega(z)=D+z C(I-zF)^{-1}C^*,\quad z\in\dC\setminus\{(-\infty,-1]\cup[1,+\infty)\}.
\]
The Schur-Frobenius formula \eqref{Sh-Fr1} yields the equality
\[
P_\sM(I-zT)^{-1}\uphar\sM=(I_\sM-z\Omega(z))^{-1},\quad z\in\dC\setminus\{(-\infty,-1]\cup[1,+\infty)\}.
\]


\begin{theorem}
\label{newchar} \cite{ArlHassi_2018}. Let $\Omega$ be an operator-valued
Herglotz-Nevanlinna defined in the region
$\dC\setminus\{(-\infty,-1]\cup[1,+\infty)\}$. Then the following
statements are equivalent
\begin{enumerate}
\def\labelenumi{\rm (\roman{enumi})}
\item $\Omega$ belongs to the class $\cRS(\sM)$;
\item $\Omega$ satisfies the inequality
\[
 I-\Omega^*(z)\Omega(z)-(1-|z|^2)\cfrac{\IM \Omega (z)}{\IM z}\ge
 0,\quad \IM z\ne 0;
\]

\item the function
\[
K(z,w):=I-\Omega^*(w)\Omega(z)-\cfrac{1-\bar w z}{z-\bar w}\,(\Omega(z)-\Omega^*(w)
\]
is a nonnegative kernel on the domains
$$\dC\setminus\{(-\infty,-1]\cup [1,\infty)\},\; \IM z>0,\quad\mbox{and}\quad  \dC\setminus\{(-\infty,-1]\cup [1,\infty)\},\; \IM z<0;$$

\item the following transform of $\Omega$
\begin{equation}
\label{formula3}
\Upsilon(z)=\left(zI-\Omega(z)\right)\left(I-z\Omega(z)\right)^{-1},\quad
 z\in\dC\setminus\{(-\infty,-1]\cup[1,+\infty)\}
\end{equation}
belongs to $\cRS(\sM)$.
\end{enumerate}
\end{theorem}

\begin{theorem}\label{inerdil} \cite{ArlHassi_2018}.
Let $\sM$ be a Hilbert space and let $\Omega\in\cRS(\sM)$. Then there exist a Hilbert space ${\wt\sM}$ containing $\sM$ as a subspace and a selfadjoint contraction $\wt T$ in $\wt \sM$ such that for all $z\in\dC\setminus\{(-\infty,-1]\cup[1,+\infty)\}$ the equality
\begin{equation}\label{ajhvelbk}
\Omega(z)=P_\sM(z I_{\wt\sM}+\wt T)(I_{\wt\sM}+z\wt T)^{-1}\uphar\sM
\end{equation}
holds. Moreover, the pair $\{\wt\sM,\wt T\}$ can be chosen such that $\wt T$ is $\sM$-minimal, i.e.,
\begin{equation}
\label{minim11}\cspan\{\wt T^n\sM:\;n\in\dN_0\}=\wt\sM.
\end{equation}
The function $\Omega$ is inner if and only if $\wt\sM=\sM$ in the representation \eqref{minim11}.

If there are two representations of the form \eqref{ajhvelbk} with pairs $\{\wt \sM_1, \wt T_1\}$ and $\{\wt \sM_2, \wt T_2\}$  that are $\sM$-simple, then
there exists a unitary operator $\wt U\in\bB(\wt\sM_1,\wt\sM_2)$ such that
\[
\wt U\uphar\sM=I_\sM,\quad \wt T_2\wt U=\wt U\wt T_1.
\]
\end{theorem}

%

\end{appendix}


\end{document}